\documentclass{amsart}

\usepackage[T1]{fontenc}
\usepackage[utf8]{inputenc}
\usepackage[english]{babel}

\usepackage{amsmath}
\usepackage{amssymb}
\usepackage{amsthm}
\usepackage{mathrsfs}
\usepackage{enumitem}

\usepackage[numbers]{natbib}
\usepackage{tikz}
\usetikzlibrary{fit,shapes.geometric,arrows}

\usepackage[includeheadfoot, margin=1in]{geometry}
\usepackage[bookmarksnumbered]{hyperref}
\usepackage{authblk}
\usepackage{amsaddr}
\usepackage[yyyymmdd,hhmmss]{datetime}

\DeclareMathOperator{\supp}{supp}

\DeclareMathOperator{\dom}{dom}
\DeclareMathOperator{\ran}{ran}

\newcommand{\Go}{\ensuremath {\mathcal{G}^{(0)}}}

\newcommand{\Skew}{A \rtimes_\alpha S}

\newcommand{\Gop}{\ensuremath {\mathcal{G}^{op}}}
\newcommand{\Ga}{\ensuremath {\mathcal{G}^{a}}}
\DeclareMathOperator{\so}{\mathfrak{s}}
\DeclareMathOperator{\ra}{\mathfrak{r}}

\newcounter{generalnumbering} \numberwithin{generalnumbering}{section}

\theoremstyle{plain}
\newtheorem{theorem}[generalnumbering]{Theorem}
\newtheorem{lemma}[generalnumbering]{Lemma}
\newtheorem{proposition}[generalnumbering]{Proposition}
\newtheorem{corollary}[generalnumbering]{Corollary}

\theoremstyle{definition}
\newtheorem{definition}[generalnumbering]{Definition}	
\newtheorem{remark}[generalnumbering]{Remark}
\newtheorem{example}[generalnumbering]{Example}

\makeatletter
\def\G{\@ifnextchar[{\@Gwithbrak}{\@Gwithoutbrak}}
\def\@Gwithbrak[#1]{\mathcal{G}^{(#1)}}
\def\@Gwithoutbrak{\mathcal{G}}
\makeatother

\newcommand{\ntag}{\stepcounter{generalnumbering}\tag{\arabic{section}.\arabic{generalnumbering}}}

\let\tempexists\exists
\let\tempforall\forall
\renewcommand{\exists}{\tempexists\mkern2mu}
\renewcommand{\forall}{\tempforall\mkern2mu}
\begin{document}

\title[The dynamics of partial inverse semigroup actions]{The dynamics of partial inverse semigroup actions}

\author{Luiz Gustavo Cordeiro$\hspace{0pt}^{\dagger}$}
\address{UMPA, UMR 5669 CNRS -- École Normale Supérieure de Lyon\\
46 alée d'Italie, 69364 Lyon Cedex 07, France}
\email{luizgc6@gmail.com}
\thanks{$\hspace{0pt}^{\dagger}$Supported by CAPES/Ciência Sem Fronteiras PhD scholarship 012035/2013-00, and by ANR project GAMME (ANR-14-CE25-0004).}

\author{Viviane Beuter}
\address{Departamento de Matem\'{a}tica, Universidade Federal de Santa Catarina, \\
Florian\'{o}polis, BR-88040-900, Brazil
and\\
Departamento de Matem\'{a}tica, Universidade do Estado de Santa Catarina, \\Joinville, BR-89219-710, Brazil}
\email{vivibeuter@gmail.com}



\date{\today}

\begin{abstract}
Given an inverse semigroup $S$ endowed with a partial action on a topological space $X$, we construct a groupoid of germs $S\ltimes X$ in a manner similar to Exel's groupoid of germs, and similarly a partial action of $S$ on an algebra $A$ induces a crossed product $A\rtimes S$. We then prove, in the setting of partial actions, that if $X$ is locally compact Hausdorff and zero-dimensional, then the Steinberg algebra of the groupoid of germs $S\ltimes X$ is isomorphic to the crossed product $A_R(X)\rtimes S$, where $A_R(X)$ is the Steinberg algebra of $X$. We also prove that the converse holds, that is, that under natural hypotheses, crossed products of the form $A_R(X)\rtimes S$ are Steinberg algebras of appropriate groupoids of germs of the form $S\ltimes X$. We introduce a new notion of topologically principal partial actions, which correspond to topologically principal groupoids of germs, and study orbit equivalence for these actions in terms of isomorphisms of the corresponding groupoids of germs. This generalizes previous work of the first-named author as well as from others, which dealt mostly with global actions of semigroups or partial actions of groups. We finish the article by comparing our notion of orbit equivalence of actions and orbit equivalence of graphs.
\end{abstract}

\maketitle

\section{Introduction}

Partial actions of groups on C*-algebras, initially defined for the group of integers in \cite{Exel1994} (and for general discrete groups in \cite{MR1331978}), are a powerful tool in the study of many C*-algebras associated to dynamical systems. In \cite{Dokuchaev2005}, Dokuchaev and Exel introduced the analogous notion of partial group actions in a purely algebraic context, and although the theory is not at present as well-developed as its C*-algebraic counterpart, it has attracted interest or researchers in the area, since some important classes of algebra, such as graph and ultragraph Leavitt path algebras, have been shown to be crossed products (see \cite{Goncalves2014a, 1706.03628v2}).

In fact, in \cite[Theorem~4.2]{Exel1998} it is proven that partial group actions correspond to actions of certain ``universal'' inverse semigroups, which were already considered in \cite{Sieben1997} and can be used, for example, to describe groupoid C*-algebras as crossed products by inverse semigroups (see \cite[Theorem~3.3.1]{Paterson1999}). Although these approaches are similar in some respects, each of them has its advantages and drawbacks -- for example, actions of inverse semigroups respect the operation completely, whereas groups have, overall, a better algebraic structure.

Groupoids are also being extensively used in order to classify and study similar classes of algebras (see \cite{Clark2014} for example), and one can relate these two approaches in the following manner: From a partial group action on a topological space we associate a transformation groupoid, or from an inverse semigroup action on a space we associate a groupoid of germs (see \cite{Abadie2004} and \cite{Exel2008}, respectively). It turns out that both in the purely algebraic and the C*-algebraic settings, the algebras of such groupoids coincide with the algebras induced from the group or semigroup actions (see \cite{Beuter2018, Demeneghi2017}). In fact, under appropriate assumptions, the relationships between the representation theory of groupoids and inverse semigroups have also been made categorical, see for example \cite{MR2969047, MR3077869, Starling2018}.

In this article we will be concerned with partial actions of inverse semigroups, defined in \cite{MR3231479}, which are a common generalization of both partial actions of groups and actions of inverse semigroups. In particular, we generalize the constructions of groupoids of germs for topological partial actions, and of crossed products for algebraic partial actions. 

Therefore we have a common ground for the study of both partial group actions and inverse semigroup actions.

The first problem we tackle is to describe the Steinberg algebra of the groupoid of germs of a topological partial inverse semigroup action as a crossed product algebra. This generalizes results of \cite{Beuter2018,Demeneghi2017}, where such isomorphisms were obtained under (strictly) stronger assumptions. In the converse direction, by starting with an appropriate crossed product, we manage to construct a groupoid of germs which realizes the isomorphism above.

Orbit equivalence and full groups for Cantor systems were initially studied by Giordano, Putnam and Skau in \cite{MR1363826, MR1710743}, and has enjoyed recent developments in \cite{MR3789176,Li2017,MR3614030}. The notion of continuous orbit equivalence can be immediately extended to partial inverse semigroup actions. We introduce and study a natural notion of topological principality for partial inverse semigroup actions, which corresponds to topological principality of the groupoid of germs. In the Hausdorff setting, we prove that two ample, topologically principal partial inverse semigroup actions are continuously orbit equivalent if and only if the corresponding groupoids of germs are isomorphic, thus generalizing the analogous part of \cite[Theorem 2.7]{Li2017}. It is important to note that the semigroups considered do not need to be isomorphic, since continuous orbit equivalence deals mostly with the dynamics of the unit space inherited from the partial action.

We finish this article by connecting the notions of continuous orbit equivalence of (partial) semigroup actions, continuous orbit equivalence of graphs, and isomorphism of Leavitt path algebras.

\subsection*{Acknowledgements}
We thank Thierry Giordano for the several useful comments and suggestions, and who worked as both the second named author's PhD supervisor (co-supervised by Vladimir Pestov), and as the first named author's supervisor during her stay at the University of Ottawa as a Visiting Student Researcher during the Winter Session of 2018. 

\section{Preliminaries}
 
\subsection*{Inverse semigroups}

A \emph{semigroup} is a set endowed with an associative binary operation $(s,t)\mapsto st$, called \emph{product}. An \emph{inverse semigroup} is a semigroup $S$ such that for every $s\in S$, there exists a unique $s^*\in S$ such that $ss^*s=s$ and $s^*ss^*=s^*$. We call $s^*$ the \emph{inverse} of $S$.

A \emph{sub-inverse semigroup} of an inverse semigroups $S$ is a nonempty subset $P\subseteq S$ which is closed under product and inverses. \emph{Homomorphisms} and \emph{isomorphisms} of inverse semigroups are defined in the same manner as for groups. We refer to \cite{MR1455373} for details.

\begin{example}
    Given a set $X$, define $\mathcal{I}(X)$ to be the set of partial bijections of $X$, i.e., bijections $f\colon\dom(f)\to\ran(f)$ where $\dom(f),\ran(f)\subseteq X$. We endow $\mathcal{I}(X)$ with the natural composition of partial maps: given $f,g\in\mathcal{I}(X)$, the product $gf$ has domain $\dom(gf)=f^{-1}(\ran(f)\cap\dom(g))$ and range $\ran(gf)=g(\ran(f)\cap\dom(g))$, and is defined by $(gf)(x)=g(f(x))$ for all $x\in\dom(gf)$.
    
    This makes $\mathcal{I}(X)$ into an inverse semigroup, where the inverse element of $f\in\mathcal{I}(X)$ is the inverse function $f^*=f^{-1}$.
\end{example}
    
Given an inverse semigroup $S$, we denote by $E(S)=\left\{e\in S:e^2=e\right\}$ the set of idempotents of $S$. $E(S)$ is a commutative sub-inverse semigroup of $S$, and it is a semilattice (Definition \ref{def:weaksemilattice}) under the order $e\leq f\iff e=ef$. This order is extended to all of $S$ by setting $s\leq t\iff s=ts^*s$. This order is preserved under products and inverses of $S$. Homomorphisms of inverse semigroups preserve their orders.

\subsection*{Partial actions of inverse semigroups}

\begin{definition}[{\cite[Definition 2.11, Proposition 3.1]{MR3231479}}]\label{def:partialhomomorphism}
A \emph{partial homomorphism} between inverse semigroups $S$ and $T$ is a map $\varphi\colon S \rightarrow T$ such that for all $s$ and $t$ in $S$, one has that
\begin{enumerate}\setlength\itemsep{1ex}
\item[(i)] $\varphi(s^*)=\varphi(s)^*$;  
\item[(ii)] $\varphi(s)\varphi(t) \leq \varphi(st)$;
\item[(iii)]\label{def:partialhomomorphism3} $\varphi(s) \leq \varphi(t)$ whenever $s \leq t$.
\end{enumerate}
Note that homomorphisms of inverse semigroups are also partial homomorphisms.
\end{definition}

In the most general context (\cite[Definition 3.3]{MR3231479}), a partial action of a semigroup $S$ on a set $X$ is simply a partial homomorphism $S\to\mathcal{I}(X)$. However, when $X$ has some extra structure (topological and/or algebraic) we will be interested in partial actions that preserve this structure.

\begin{definition}
A \emph{topological partial action} of an inverse semigroup $S$ on a topological space $X$ is a tuple $\theta=(\left\{X_s\right\}_{s\in S},\left\{\theta_s\right\}_{s\in S})$ such that:
\begin{enumerate}[label=(\roman*)]
\item For all $s\in S$, $X_s$ is an open subset of $X$ and $\theta_s\colon X_{s^*}\to X_s$ is a homeomorphism;
\item The map $s\mapsto\theta_s$ is a partial homomorphism of inverse semigroups;
\item $X=\bigcup_{e\in E(S)}X_e$.
\end{enumerate}
If $s\mapsto\theta_s$ is a homomorphism of inverse semigroups, we call $\theta$ a \emph{global action}, or simply an \emph{action} if no confusion arises.
\end{definition}

Condition (iii) above is usually called \emph{non-degeneracy}, and is sometimes not required. If (i) and (ii) are satisfied by a tuple $\theta$ as above, then $X_{s^*}\subseteq X_{s^*s}$ for all $s\in S$, and thus one can always substitute $X$ by $\bigcup_{e\in E(S)}X_e$ (which in fact coincides with $\bigcup_{s\in S}X_s$) and obtain a non-degenerate partial action. In fact, $\theta$ is a global action if and only if $X_{s^*}=X_{s^*s}$ for all $s\in S$. Similar comments hold for partial actions of groups on algebras, which we now define. For the remainder of this section, we fix a commutative unital ring $R$, and will consider algebras over $R$.

\begin{remark}
Every ring has a canonical $\mathbb{Z}$-algebra structure, or alternatively, when restricted to commutative rings, every commutative ring $R$ has a canonical $R$-algebra structure. Thus the definitions we adopt for algebras restrict to partial actions and crossed products of rings.
\end{remark}
\begin{definition}
An \emph{algebraic partial action} of an inverse semigroup $S$ on an associative $R$-algebra $A$ is a tuple $\alpha=(\left\{A_s\right\}_{s\in S},\left\{\alpha_s\right\}_{s\in S})$ such that:
\begin{enumerate}
\item[(i)] For all $s\in S$, $A_s$ is an ideal of $A$ and $\alpha_s\colon A_{s^*}\to A_s$ is an $R$-isomorphism;
\item[(ii)] $\alpha\colon S\to\mathcal{I}(A)$, $s\mapsto\alpha_s$ is a partial homomorphism of inverse semigroups;
\item[(iii)] $X=\operatorname{span}_R\bigcup_{e\in E(S)}A_e$.
\end{enumerate}
If $s\mapsto\alpha_s$ is a homomorphism of inverse semigroups, we call $\alpha$ a \emph{global action} or simply an \emph{action}.
\end{definition}

\subsection*{Crossed products}

Let $R$ be a commutative unital ring, and let $\alpha=(\{A_s\}_{s\in S}, \{\alpha_s\}_{s \in S})$ be a partial action of an inverse semigroup $S$ on an associative $R$-algebra $A$. Consider $\mathscr{L}=\mathscr{L}(\alpha)$ the $R$-module of all finite sums of the form 
\[\sum_{s\in S}^{\text{finite}} a_s \delta_s, \qquad \text{where } a_s \in A_s \text{ and } \delta_s  \text{ is a formal symbol},\]
with a multiplication defined as the bilinear extension of the rule
\[(a_s \delta_s)(b_t \delta_t) = \alpha_{s}(\alpha_{s^*}(a_s) b_t) \delta_{st}.\]

Then $\mathscr{L}$ is an $R$-algebra which is possibly not associative (see \cite[Example~3.5]{Dokuchaev2005}). A proof similar to that of \cite[Corollary 3.2]{Dokuchaev2005} shows that if $A_s$ is idempotent or non-degenerate for each $s\in S$, then $\mathscr{L}$ is associative.

\begin{definition}\label{def:partialskewinversesemigroupring}
Let $\alpha=(\left\{A_s\right\}_{s\in S},\left\{\alpha_s\right\}_{s\in S})$ be an algebraic partial action of an inverse semigroup $S$ on an $R$-algebra $A$ end let $\mathscr{N}=\mathscr{N}(\alpha)$ be the additive subgroup of $\mathscr{L}$ generated by all elements of the form 
\[a\delta_r - a\delta_s,\quad\text{where}\quad r\leq s\quad\text{and}\quad a\in A_r.\] 
Then $\mathscr{N}$ is an ideal of the $R$-algebra $\mathscr{L}$. We define the \emph{crossed product}, which we denote by $A\rtimes_\alpha S$ (or simply $A\rtimes S$) as the quotient algebra
\[A\rtimes_\alpha S:=\mathscr{L}/\mathscr{N}\]
The class of an element $x\in\mathscr{L}$ in $\Skew$ will be denoted by $\overline{x}$.
\end{definition}

\begin{remark}
\begin{enumerate}
    \item As a ring, $A\rtimes S$ depends only on the ring structure of $A$ and the maps $\alpha$. So distinct algebra structures over $A$ will induce distinct algebra structures over the \emph{same ring} $A\rtimes S$ (as long as the partial action preserves these distinct algebra structures).
    \item Crossed products are sometimes called \emph{skew inverse semigroup algebras} or \emph{rings}, or \emph{partial crossed products} (see \cite{beutergoncalvesoinertroyer2018,Boava2013,Dokuchaev2005}). Since these are simply particular cases of the construction above, we adopt the simplest nomenclature for the most general case.
\end{enumerate}
\end{remark}

The \emph{diagonal} of the crossed product $A\rtimes S$ is the additive abelian subgroup generated by elements of the form $\overline{a\delta_e}$, where $e\in E(S)$ and $a\in A_e$, and the diagonal is a subalgebra of $A\rtimes S$.

Recall that a ring $B$ is \emph{left $s$-unital} if for all finite subsets $F\subseteq B$, there exists $u\in B$ such that $x=ux$ for all $x\in F$.

\begin{proposition}
Suppose that $\alpha$ is a partial action of $S$ on an algebra $A$, and that for all $e\in E(S)$, $A_e$ is a left $s$-unital ring. Then $A$ is isomorphic to the diagonal algebra of $A\rtimes S$.
\end{proposition}
\begin{proof}
Any element of $A$ is a sum of elements of $\bigcup_{e\in E(S)}A_e$, so the same argument of \cite[Theorem 1]{MR0419511} proves that $A$ is a left $s$-unital ring. The proof of \cite[Proposition 4.3.11]{cordeirothesis} (see also \cite[Proposition 3.1]{beutergoncalvesoinertroyer2018}) can be easily adapted to obtain an isomorphism between $A$ and the diagonal subalgebra of $A\rtimes S$.\qedhere
\end{proof}

\subsection*{Étale Groupoids}

A \emph{groupoid} is a small category $\G$ with invertible arrows. We identify $\G$ with the underlying set of arrows, so that objects of $\G$ correspond to unit arrows, and the space of all units is denoted by $\G[0]$. The \emph{source} of an element $a\in\G$ is defined as $\so(a)=a^{-1}a$ and the \emph{range} of $a$ is $\ra(a)=aa^{-1}$. A pair $(a,b)\in \G\times\G$ is \emph{composable} (i.e., the product $ab$ is defined) if and only if $\so(a)=\ra(b)$, and the set of all composable pairs is denoted by $\G[2]$.

A \emph{topological groupoid} is a groupoid $\G$ endowed with a topology which makes the multiplication map $\G[2]\ni(a,b)\mapsto ab\in\G$ and the inverse map $\G\ni a\mapsto a^{-1}\in \G$ continuous, where we endow $\G[2]$ with the topology induced from the product topology of $\G\times\G$.

\begin{definition}[\cite{MR3077869,MR2304314}]
An \emph{étale groupoid} is a topological groupoid $\G$ such that the source map $\mathfrak{s}\colon\G\to\Go$ is a local homeomorphism.
\end{definition}

Alternatively, a topological groupoid $\G$ is étale precisely when $\Go$ is open and the product of any two open subsets of $\G$ is open (see \cite[Theorem 5.18]{MR2304314}), where the product of $A,B\subseteq\G$ is defined as
\[AB=\left\{ab:(a,b)\in(A\times B)\cap \G^{(2)}\right\}.\]

An open \emph{bisection} of an étale groupoid is an open subset $U\subseteq\G$ such that the source and range maps are injective on $U$, and hence homeomorphisms onto their images. The set of all open bisections of an étale groupoid forms a basis for its topology and it is an inverse semigroup under the product of sets. We denote this semigroup by $\G^{op}$.

\begin{definition}
An étale groupoid is \emph{ample} if $\G[0]$ is Hausdorff and admits a basis of compact-open subsets.
\end{definition}

Suppose that $\G$ is an ample groupoid. Then $\G$ admits a basis of compact-open bisections, since $\G[0]$ does and $\so\colon\G\to\G[0]$ is a local homeomorphism. Since $\Go$ is Hausdorff then $\G^{(2)}$ is closed in $\G\times\G$, then the product of two compact subsets $A,B$ of $\G$ is compact, as $AB$ is the image of the compact $(A\times B)\cap\G[2]$ under the continuous product map (alternatively, see \cite[Lemma 3.13]{MR3077869}). We denote by $\Ga$ the sub-inverse semigroup of $\Gop$ consisting of compact-open bisections and call $\Ga$ the \emph{ample semigroup} of $\G$.

\begin{example}\label{examplecanonicalaction}
Let $\mathcal{G}$ be an étale groupoid. The \emph{canonical action} of $\Gop$ on $\mathcal{G}^{(0)}$ is defined as $\tau=\tau^{\G}=\left(\left\{\ra(U)\right\}_{U\in\Gop},\left\{\tau_U\right\}_{U\in\Gop}\right)$, with $\tau_U\colon\so(U)\to\ra(U)$ the homeomorphism $\tau=\ra\circ\so|_U^{-1}$. This is the homeomorphism which takes the source of each arrow of $U$ to its range.
\end{example}

\subsection*{Steinberg algebras of ample groupoids}

Throughout this section, we fix a commutative unital ring $R$. Given an ample Hausdorff groupoid $\mathcal{G}$, we denote by $R^{\G}$ the $R$-module of $R$-valued functions on $\G$. Given $A\subseteq\G$, define $1_A$ as the characteristic function of $A$ (with values in $R$).

\begin{definition}
Given an ample groupoid $\G$, $A_R(\G)$ is the $R$-submodule of $R^{\G}$ generated by the characteristic functions of compact-open bisections of $\G$.
\end{definition}

The \emph{support} of $f\in R^{\G}$ is defined as $\supp f=\left\{a\in\G:f(a)\neq 0\right\}$. If $\G$ is Hausdorff, then $A_R(\G)$ coincides with the $R$-module of locally constant compactly supported $R$-valued functions on $\G$ \cite[Lemma 3.3]{Clark2014}.

In the general (non-Hausdorff) case, for every $f\in A_R(\G)$ and every $x\in\Go$, $(\supp f) \cap \so^{-1}(x)$ and $(\supp f)\cap\ra^{-1}(x)$ are finite, and so we can define their \emph{convolution product}
\[(f\ast g)(a)=\sum_{xy=a}f(x)g(y)=\sum_{x\in\ra^{-1}(\ra(a))}f(x)g(x^{-1}a)=\sum_{y\in\so^{-1}(\so(a))}f(ay^{-1})g(y).\]

This product makes $A_R(\G)$ an associative $R$-algebra, called the \emph{Steinberg algebra} of $\G$ (with coefficients in $R$).

The map $\Ga\to A_R(\G)$, $U\mapsto 1_U$, is a representation of $\Ga$ as a Boolean semigroup (see \cite{MR3077869}), that is, it satisfies (i) $1_U\ast1_V=1_{UV}$; and (ii) $1_{U\cup V}=1_U+1_V$ if $U\cap V=\varnothing$ and $U\cup V\in\Ga$.

In fact, $A_R(\G)$ is universal for such representations. The proof for a general commutative ring with unit $R$ follows the same arguments as in \cite[Theorem~3.10]{Clark2014}, and we state it here explicitly:

\begin{theorem}[Universal property of Steinberg algebras, {\cite[Theorem 4.4.8]{cordeirothesis}}]\label{theo:universalpropertyofsteinbergalgebras}
Let $R$ be a commutative unital ring and $\G$ an ample Hausdorff groupoid. Then $A_R(\G)$ is universal for \emph{Boolean representations} of $\Ga$, i.e., if $B$ is an $R$-algebra and $\pi\colon\Ga\to B$ is a function satisfying
\begin{enumerate}[label=(\roman*)]
    \item $\pi(AB)=\pi(A)\pi(B)$ for all $A,B\in\Ga$; and
    \item $\pi(A)=\pi(A\setminus B)+\pi(B)$ whenever $A,B\in\Ga$ and $B\subseteq A$,
\end{enumerate}
then there exists a unique $R$-algebra homomorphism $\Phi\colon A_R(\G)\to B$ such that $\Phi(1_U)=\pi(U)$ for all $U\in\Ga$.
\end{theorem}

Recall that a topological space $X$ is \emph{zero-dimensional} (or has \emph{small inductive dimension} $0$) if it admits a basis of clopen subsets of $X$. A locally compact Hausdorff space is zero-dimensional if and only if it is totally disconnected. Moreover, an étale groupoid $\G$ is ample if and only if $\G[0]$ is locally compact Hausdorff and zero-dimensional.

\begin{example}
Every locally compact Hausdorff and zero-dimensional space $X$ is an ample groupoid with $X^{(0)}=X$ (that is, the product is only defined as $xx=x$ for all $x\in X$). The Steinberg algebra $A_R(X)$ coincides with the $R$-algebra of locally constant compactly supported $R$-valued functions on $X$ with pointwise operations.
\end{example}

In general, we identify $A_R(\Go)$ with the sub-$R$-algebra $D_R(\G)=\left\{f\in A_R(\G):\supp f\subseteq\G[0]\right\}$ of $A_R(\G)$, called the \emph{diagonal subalgebra} of $A_R(\G)$. More precisely, the map $D_R(\G)\ni f\mapsto f|_{\G[0]}\in A_R(\G[0])$ is an $R$-algebra-isomorphism, and its inverse extends every $f\in A_R(\Go)$ as $0$ on $\G\setminus\Go$.

\section{Groupoids of germs}

Groupoids of germs were already considered by Paterson in \cite{Paterson1999} for localizations of inverse semigroups, and for natural actions of pseudogroups by Renault in \cite{Renault2008}. In \cite{Exel2008}, Exel defined groupoids of germs for arbitrary actions of inverse semigroups on topological spaces in a similar, albeit more general, manner than both previous definitions of groupoids of germs. Moreover, partial actions of groups -- also introduced by Exel in \cite{Exel1994} -- have many application in the theory of C*-dynamics, see for example \cite{Boava2013,MR3539347,MR1703078,MR3699170,Goncalves2014c,Goncalves2017}. Partial group actions also induce transformation groupoids similarly to the classical (global) case, see \cite{Abadie2004}.

Our objective in this section is to construct a groupoid of germs associated to any partial action of an inverse semigroup in a way that generalizes both groupoids of germs of inverse semigroup actions, and transformation groupoids of partial group actions.

Let $\theta=\left(\{X_s\}_{s\in S},\{\theta_s\}_{s\in S}\right)$ be a partial action of an inverse semigroup $S$ on a topological space $X$. We denote by $S\ast X$ the subset of $S \times X$ given by
\[S\ast X:= \left\lbrace (s,x) \in S \times X : x \in X_{s^*} \right\rbrace.\]

Recall that a \emph{semigroupoid} is a structure satisfying the same axioms as a category\footnote{A slightly more general definition appears in \cite[Definition 2.1]{MR2754831}.}, except possibly the existence of identities at objects (see \cite[Appendix B]{MR915990}). Quotients of semigroupoids are defined, up to obvious modifications, in the same manner as quotients of categories (see \cite[Section II.8]{MR1712872}).

We make $S\ast X$ a semigroupoid with object space $(S\ast X)^{(0)}=X$ by setting the source and range maps as
\[\so(s,x)=x,\quad\ra(s,x)=\theta_s(x)\]
and the product $(s,x)(t,y)=(st,y)$ whenever $\so(s,x)=\ra(t,y)$. Moreover, $S\ast X$ is an \emph{inverse semigroupoid}, in the sense that for every $p=(s,x)\in S\ast X$, $p^*=(s^*,\theta_s(x))$ is the unique element of $S\ast X$ satisfying $pp^*p=p$ and $p^*pp^*=p^*$.

We define the \emph{germ relation} $\sim$ on $S\ast X$: for every $(s,x)$ and $(t,y)$ in $S\ast X$,
\[\ntag\label{eq:equivalencegroupoidgerms}
(s,x) \sim (t,y)\iff x=y\quad\text{and there exists } u \in S\quad\text{such that}\quad u \leq s,t\quad\text{and}\quad x \in X_{u^*}.\]
Alternatively,
\[\ntag\label{eq:equivalencegroupoidgerms2}
(s,x) \sim (t,y)\iff x=y\quad\text{and there exists } e \in E(S)\quad\text{such that}\quad x \in X_e\quad\text{and}\quad se=te.\]
Indeed, if $(s,x) \sim (t,y)$ and $u\in S$ satisfies \eqref{eq:equivalencegroupoidgerms}, then $e=u^*u$ satisfies \eqref{eq:equivalencegroupoidgerms2}. Conversely, if $e\in E(S)$ satisfies \eqref{eq:equivalencegroupoidgerms2}, then $u=se$ satisfies \eqref{eq:equivalencegroupoidgerms}.

We call the $\sim$-equivalence class of $(s,x)$ is the \emph{germ of $s$ at $x$}, and we denote it by $[s,x]$.

\begin{remark}\label{changebygreater}
If $u\leq s$ in $S$ and $x\in X_{u^*}$, then $x\in X_{s^*}$ as well and $[s,x]=[u,x]$.
\end{remark}

\begin{proposition}\label{prop:quotientoftrnsformationsemigroupoidisagroupoid}
The relation $\sim$ is a congruence, and the quotient semigroupoid $S\ltimes X:=(S\ast X)/\!\!\sim$ is a groupoid. The inverse of $[s,x]\in S\ltimes X$ is $[s^*,\theta_s(x)]$.
\end{proposition}
\begin{proof}
The relation $\sim$ is clearly reflexive and symmetric. As for transitivity, if $(s,x)\sim (t,y)$ and $(t,y)\sim (r,z)$, then $x=y=z$, so there exist $u\leq s,t$ and $v\leq t,r$ such that $x\in X_{u^*}\cap X_{v^*}$. It follows that
\[uv^*v\leq tv^*v=v\leq r,\qquad\text{and of course}\quad uv^*v\leq u\leq s.\]
Moreover, $x\in X_{u^*}\cap X_{v^*v}\subseteq X_{(uv^*v)^*}$. Therefore $(s,x)\sim(r,z)$ by \eqref{eq:equivalencegroupoidgerms}.

To prove that $\sim$ is a congruence, first note that the source and range maps of $S\ast X$ are invariant on $\sim$-equivalence classes. We need to prove that $\sim$ is invariant under taking products, so suppose that $(s_i,x_i)\sim(t_i,y_i)$ ($i=1,2$) and $\so(s_1,x_1)=\ra(s_2,x_2)$. Then for $i=1,2$ we have $x_i=y_i$, and there exists $u_i\leq s_i,t_i$ such that $x_i\in X_{u_i^*}$. Thus
\[u_1u_2\leq s_1s_2,t_1t_2\qquad\text{and}\qquad x_2\in X_{u_2^*}\cap \theta_{u_2}^{-1}(X_{u_1^*}\cap X_{u_2})\subseteq X_{(u_1u_2)^*},\] because $\theta_{u_2}(x_2)=\theta_{s_2}(x_2)=x_1$. This proves that
\[(s_1,x_1)(s_2,x_2)=(s_1s_2,x_2)\sim (t_1t_2,y_2)=(t_1,y_1)(t_2,y_2).\]

To conclude that $S\ltimes X$ is a groupoid, simply note that for all $(s,x)\in S\ast X$, if $(s,x)(t,y)$ is defined, then
\[([s^*,\theta_s(x)][s,x])[t,y]=[s^*s,x][t,y]=[s^*st,y]=[t,y],\]
by Remark \ref{changebygreater}, and similarly $[r,z]([s,x][s^*,\theta_s(x)])=[r,z]$ whenever $(r,z)(s,x)$ is defined. This means precisely that $S\ltimes X$ is a groupoid.
\end{proof}

Note that, by construction, the object (unit) space of $S\ltimes X$ is $X$, which we identify with the set of unit arrows of $S\ltimes X$ as usual: the source $\so[s,x]=x$ of an arrow $[s,x]\in S\ltimes X$ corresponds to the arrow
\[[s,x]^{-1}[s,x]=[s^*,\theta_s(x)][s,x]=[s^*s,x]\]
and similarly the range $\ra[s,x]=\theta_s(x)$ corresponds to the arrow $[ss^*,\theta_s(x)]$. In other words, we identify $(S\ltimes X)^{(0)}$ and $X$ via the map
\[\ntag\label{eq:unitspaceofgroupoidofgerms}
X\to (S\ltimes X)^{(0)},\qquad x\mapsto [e,x],\quad\text{where}\quad e\in E(S)\quad\text{is chosen so that}\quad x\in X_e\]
which is well-defined since we only consider non-degenerate partial actions.

We will now endow $S\ltimes X$ with an appropriate topology. Given $s\in S$ and $U\subseteq X_{s^*}$, define the subset $[s,U]$ of $S\ltimes X$
\[[s,U]=\left\{[s,x]:x\in U\right\}.\]
Using the definition of germs, it readily follows that
\[\ntag\label{eq:intersectionofbasic}
[s,U]\cap[t,V]=\bigcup\left\{[z,U\cap V\cap X_{z^*}]:z\in S,\ z\leq s,t\right\}.\]

\begin{proposition}\label{propositionbgermisbasis}
The family $\mathscr{B}_{\mathrm{germ}}$ of sets $[s,U]$, where $s\in S$ and $U\subseteq X_{s^*}$ is open, forms a basis for a topology on $S\ltimes X$, which makes it a topological groupoid.
\end{proposition}

\begin{proof}
By \eqref{eq:intersectionofbasic}, $\mathscr{B}_{\mathrm{germ}}$ is a basis for a topology on $S\ltimes X$.

Let $m\colon(S\ltimes X)^{(2)}\to S\ltimes X$ be the product map. Given $[u,U]\in\mathscr{B}_{\mathrm{germ}}$, let us prove that
\[\ntag\label{eq:continuityofproduct}
m^{-1}[u,U]=\bigcup\left\{([s,X_{s^*}]\times[t,U\cap X_{t^*}])\cap (S\ltimes X)^{(2)}:s,t\in S\quad\text{and}\quad st\leq u\right\}\]
Indeed, the inclusion ``$\supseteq$'' in \eqref{eq:continuityofproduct} is immediate from the definition of the product. For the converse inclusion, assume $([s,y],[q,x])\in m^{-1}[u,U]$. This means that $[sq,x]\in [u,U]$, so $x\in U$ and there exists $v\leq sq,u$ such that $x\in X_{v^*}$. Set $t=qv^*v$, so that $x\in X_{t^*}$, $[q,x]=[t,x]$, and $st=sqv^*v=v\leq u$, and therefore $([s,y],[q,x])=([s,y],[t,x])$ belongs to the right-hand side of \eqref{eq:continuityofproduct}. This proves that the product map is continuous.

Similarly, the definition of the inverse in $S\ltimes X$ implies that
\[[u,U]^{-1}=[u^*,\theta_u(U)]\in\mathscr{B}_{\mathrm{germ}},\]
so the inversion map is also continuous.\qedhere
\end{proof}

\begin{proposition}\label{prop:bisectionhomeomorphic}
$S\ltimes X$ is an étale groupoid, and each basic open set $[s,U]\in\mathscr{B}_{\mathrm{germ}}$ is an open bisection.
\end{proposition}

\begin{proof}
Given $[s,U]\in\mathscr{B}_{\mathrm{germ}}$, we have $\so[s,U]=[s^*s,U]\in\mathscr{B}_{\mathrm{germ}}$, so the source map is open. Moreover, it is injective on $[s,U]$. Therefore the source map is locally injective, continuous and open, hence a local homeomorphism, so $S\ltimes X$ is étale. Similarly, the range map is also injective on $[s,U]$, which is therefore a bisection.\qedhere
\end{proof}

The proof above shows that a basic open set of $(S\ltimes X)^{(0)}$ is of the form $[e,U]$, where $e\in E(S)$ and $U\subseteq X_e$. Under the identification of $(S\ltimes X)^{(0)}$ with $X$ as in Equation \eqref{eq:unitspaceofgroupoidofgerms}, $[s,U]$ corresponds to $U$. Therefore this identification is a homeomorphism.

Notice that if $s \in S$ and $U \subseteq X_{s^*}$ is an open set then $[s,U]$ is compact if and only if $U$ is compact. Moreover, if $\mathscr{B}$ is a basis for the topology of $X$, then a basis for $S\ltimes X$ consists of those sets of the form $[s,U]$ with $U\in \mathscr{B}$. Hence, if $X$ is zero-dimensional then the collection of sets of the form $[s,U]$ with $U$ compact-open subset of $X$ is a basis for $S\ltimes X$.

\begin{corollary}
If $X$ is a locally compact Hausdorff and zero-dimensional space then $S\ltimes X$ is an ample groupoid.
\end{corollary}

Let us prove a universal property for the groupoid of germs. Recall from Example \ref{examplecanonicalaction} the definition of the \emph{canonical action} of an étale groupoid.

\begin{theorem}[Universal property of groupoids of germs]\label{theo:universalpropertyofgroupoidofgerms}
Let $\theta$ be a topological partial action of an inverse semigroup $S$ on a topological space $X$. Suppose that $\G$ is an étale groupoid, $\sigma\colon S\to\G^{op}$ is a partial homomorphism, and $\phi\colon X\to\G[0]$ is a continuous function satisfying
\begin{enumerate}[label=(\roman*)]
    \item $\phi(X_s)\subseteq\ra(\sigma(s))$ for all $s\in S$; and
    \item\label{universalpropertyofgroupoidofgerms2} $\tau_{\sigma(s)}(\phi(x))=\phi(\theta_s(x))$ for all $s\in S$ and $x\in X_{s^*}$,
\end{enumerate}
where $\tau$ denotes the canonical action of $\G^{op}$ on $\G[0]$.

Then there exists a unique continuous groupoid homomorphism $\Psi\colon S\ltimes X\to\G$ satisfying
\[\ntag\label{eq:universalpropertygroupoidofgerms}
\Psi[s,x]\in\sigma(s)\quad\text{and}\quad\so(\Psi[s,x])=\phi(x),\]
whenever $s\in S$ and $x\in X_{s^*}$.
\end{theorem}
\begin{proof}
Equation \eqref{eq:universalpropertygroupoidofgerms} simply means that $\Psi[s,x]=\so|_{\sigma(s)}^{-1}(\phi(x))$ for all $s\in S$ and $x\in X_{s^*}$, so uniqueness is immediate.

Define $\Phi\colon S\ast X\to\G$ by $\Phi(s,x)=\so|_{\sigma(s)}^{-1}(\phi(\phi(x)))$ for all $(s,x)\in S\ast X$, that is, $\Phi(s,x)$ is the arrow in $\sigma(s)$ with source $\phi(x)$. Let us prove that it is a semigroupoid homomorphism.

Suppose that the product $(s,x)(t,y)$ is defined in $S\ast X$. This means that $x=\theta_t(y)$. Applying $\phi$ and using property (\ref{universalpropertyofgroupoidofgerms2}) yields
\[\so(\Phi(s,x))=\phi(x)=\phi(\theta_t(y))=\tau_{\sigma(t)}(\phi(y)),\]
and the last term above is simply the range of the arrow in $\sigma(t)$ whose source is $\phi(y)$, that is,
\[\so(\Phi(s,x))=\ra(\Phi(t,y)).\]
Therefore, the product $\Phi(s,x)\Phi(t,y)$ is defined. It belongs to $\sigma(s)\sigma(t)\sigma(st)$, since $\sigma$ is a partial homomorphism, and its source is $\so(\Phi(t,y))=\phi(y)$. Therefore,
\[\Phi(s,x)\Phi(t,y)=\Phi(st,y)=\Phi((s,x)(t,y)),\]
which proves that $\Phi$ is a semigroupoid homomorphism.

Let us prove that $\Phi$ is invariant by the germ relation $\sim$ as in \eqref{eq:equivalencegroupoidgerms}: Suppose $(s,x)\sim(t,y)$. Then $x=y$ and there exists $v\leq s,t$ such that $x\in X_{v^*}$. Then $\Phi(v,x)$ is an arrow which in $\sigma(v)\subseteq\sigma(s),\sigma(t)$, as $\sigma$ is a partial homomorphism, and whose source is $\phi(x)$, thus
\[\Phi(s,x)=\Phi(v,x)=\Phi(t,x)=\Phi(t,y).\]
Therefore $\Phi$ factors though a groupoid homomorphism $\Psi\colon S\ltimes X\to\G$ satisfying \eqref{eq:universalpropertygroupoidofgerms}.

It remains only to prove that $\Psi$ is continuous. Suppose that $V\subseteq\G$ is open. As $\G$ is étale, $\so(V)$ is open. We are thus finished by proving that
\[\ntag\label{eq:proofofcontinuityuniversalpropertygroupoidofgerms}\Psi^{-1}(V)=\bigcup\left\{[s,X_{s^*}\cap\phi^{-1}(\so(\sigma(s)\cap V))]:s\in S\right\}.\]
If $[s,x]\in\Psi^{-1}(V)$, then \[\phi(x)=\so(\Psi[s,x])\in\so(\sigma(s)\cap V)\]
so $[s,x]$ belongs to the right-hand side of \eqref{eq:proofofcontinuityuniversalpropertygroupoidofgerms}.

Conversely, if $[s,x]$ belongs to the right-hand side of \eqref{eq:proofofcontinuityuniversalpropertygroupoidofgerms}, then there is an arrow $\gamma$ in $\sigma(s)\cap V$ whose source is $\phi(x)$. By definition of $\Psi$, we have $\Psi[s,x]=\gamma\in V$, so $[s,x]\in\Psi^{-1}(V)$.\qedhere
\end{proof}
\begin{example}
Following \cite{Paterson1999}, a \emph{localization} consists of a global action $\theta=(\left\{X_s\right\}_{s\in S},\left\{\theta_s\right\}_{s\in S})$ of an inverse semigroup $S$ on a topological space $X$ such that $\left\{X_s:s\in S\right\}$ is a basis for the topology of $S$. The groupoid of germs in the sense of Paterson (see \cite{Paterson1999}) coincides with the definition above of groupoids of germs.
\end{example}

\begin{example}
Let $X$ be a topological space. The \emph{canonical action} of $\mathcal{I}(X)$ on $X$ is the action $\tau$ given by $\tau_\phi=\phi$ for all $\phi\in\mathcal{I}(X)$. A \emph{pseudogroup} on $X$ is a sub-inverse semigroup of $\mathcal{I}(X)$ whose elements are homeomorphisms between open subsets of $X$.

Let $\mathcal{B}$ be a basis for the topology of $X$, and for each $B\in\mathcal{B}$ consider its identity function $\operatorname{id}_B\colon B\to B$.

Given a pseudogroup $\mathcal{G}$ on $X$, let $\mathcal{G}\mathcal{B}$ be the sub-inverse semigroup of $\mathcal{I}(X)$ generated by $\mathcal{G}\cup\left\{\operatorname{id}_B:B\in\mathcal{B}\right\}$, which is again a pseudogroup on $X$, and in fact the canonical action of $\mathcal{G}\mathcal{B}$ on $X$ is a localization.

The groupoid of germs in the sense of Renault (see \cite{Renault2008}) of $\mathcal{G}$ coincides with the groupoid of germs $\mathcal{G}\mathcal{B}\ltimes X$ defined above.
\end{example}

\begin{example}[Transformation groupoids]
In the case that $S$ is a discrete group, the equivalence relation $\sim$ on $S \ast X$ is trivial and the topology is the product topology, so $S\ltimes X$ is (isomorphic to) the usual transformation groupoid.
\end{example}

\begin{example}[Maximal group image]\label{example:maximalgroupimageastransformationgroupoid}
Suppose that $X=\left\{x\right\}$ is a one-point space on which $S$ acts trivially. Then $S\ltimes X$ is a groupoid whose unit space is $X$, a singleton, that is, $S\ltimes X$ is a group. The universal property of $S\ltimes X$ implies that $S\ltimes X$ satisfies the universal property of the maximal group image $\mathbf{G}(S)$ of $S$ (see \cite{Paterson1999} or Section~\ref{sectionassociatedactionsandisomorphicgroupoidofgerms}), so $S\ltimes X$ is isomorphic to $\mathbf{G}(S)$.
\end{example}

\begin{example}[Restricted product groupoid]
Let $X=E(S)$ with the discrete topology, and let $\theta=\left(\{X_s\}_{s\in S},\{\theta_s\}_{s\in S}\right)$ be the \emph{Munn representation} of $S$ (see \cite{MR0262402}): $X_s=\left\{e\in E(S):e\leq ss^*\right\}$ and $\theta_s(e)=ses^*$ for all $e\in X_{s^*}$.

From $S$ we can construct the \emph{restricted product groupoid} $(S,\cdot)$, which is the same as $S$ but the product $s\cdot t=st$ is defined only when $s^*s=tt^*$. (See \cite{MR1694900} for more details.)

Then $S\ltimes E(S)$ is a discrete groupoid, and the map
\[S\ltimes E(S)\to(S,\cdot),\qquad [s,e]\mapsto se\]
is an isomorphism from $S\ltimes E(S)$ to $(S,\cdot)$, with inverse $s\mapsto [s,s^*s]$.
\end{example}

\begin{example}[{\cite{Steinberg2010}}]\label{nonhausdorffaction}
Let $S=\mathbb{N}\cup\{\infty, z\}$, with product given, for $m,n\in\mathbb{N}$,
\[nm=\min(n,m),\quad n\infty=\infty n=nz=zn=n,\qquad z\infty=\infty z=z\quad\text{and}\quad zz=\infty\infty=\infty.\]
In other words, $S$ is the inverse semigroup obtained by adjoining the lattice $\mathbb{N}$ to the group $\left\{\infty,z\right\}$ of order 2 (where $\infty$ is the unit), in a way that every element of $\mathbb{N}$ is smaller than $z$ and $\infty$.

Let $X=E(S)=\mathbb{N}\cup\left\{\infty\right\}$, seen as the one-point compactification of the natural numbers, and $\theta$ the Munn representation of $S$, so that $S\ltimes X=(S,\cdot)$, however with the topology whose open sets are either cofinite or contained in $\mathbb{N}$. In particular, $S \ltimes X$ is not Hausdorff.
\end{example}

\begin{example}\label{ex:amplegermesopenbisections}
Every étale groupoid is isomorphic to a groupoid of germs. Indeed, let $\G$ be an étale groupoid, and $S$ any subsemigroup of $\G^{op}$ which covers $\G$ (i.e., $\G=\bigcup_{A\in S}A$). We let $S$ act on $\G[0]$ by the restriction of the canonical action of $\G^{op}$ on $\Go$ (as Example \ref{examplecanonicalaction}). Then the map $\Phi\colon S\ltimes \Go\to\G$, $[A,x]\mapsto\so|_A^{-1}(x)$, is a surjective homomorphism of topological groupoids. Moreover, $\Phi$ is injective if and only if $S$ forms a basis for some topology on $\G$ (see \cite[Proposition 5.4]{Exel2008}).

In particular, if $\G$ is an ample groupoid, and $\gamma$ is the canonical action of $\G^{a}$ on $\Go$, then the groupoid of germs $\G^{a}\ltimes \Go$ is (canonically) isomorphic to $\G$.
\end{example}

For the results in Section \ref{sec:continuousorbitequivalence} we will need to consider partial actions with Hausdorff groupoids of germs. Let us mention conditions on inverse semigroups which guarantee that groupoids of germs are Hausdorff.

\begin{definition}\label{def:weaksemilattice}
A poset $(L,\leq)$ is a
\begin{enumerate}
\item\label{def:itemweaksemilattice} (meet-)\emph{weak semilattice} if for all $s,t\in L$ there exists a finite (possibly empty) subset $F\subseteq L$ such that
\[\left\{x\in L:x\leq s\text{ and }x\leq t\right\}=\bigcup_{f\in F}\left\{x\in L:x\leq f\right\}.\]
\item\label{def:itemsemilattice} (meet-)\emph{semilattice} if every pair of elements $s,t\in L$ admits a meet, that is, $s\land t=\inf\left\{s,t\right\}$ exists.
\end{enumerate}
\end{definition}

\begin{example}
If $\G$ is an étale groupoid, then $\Gop$ is a semilattice. If $\G$ is an ample Hausdorff groupoid, then $\Ga$ is a semilattice. In either of these cases, the meets are given by intersection: $U\land V=U\cap V$.
\end{example}

\begin{example}
Every $E$-unitary inverse semigroup $S$ (see Section \ref{sectionassociatedactionsandisomorphicgroupoidofgerms}) is a weak semilattice: If $s,t\in S$ do not have any common lower bound, $F=\varnothing$ in satisfies Definition \ref{def:weaksemilattice}\ref{def:itemweaksemilattice}. If $s,t$ have some common lower bound then they are compatible, and $s\land t=st^*t$, so instead we take $F=\left\{st^*t\right\}$.

More generally, every $E^*$-unitary inverse semigroup is a semilattice (see \cite[Example 4.5.4]{cordeirothesis}).
\end{example}

The following relation between inverse semigroups which are weak semilattices and the topology of their groupoids of germs can be proven just as in \cite[Theorem~5.17]{Steinberg2010}.

\begin{proposition}[{\cite[Theorem~5.17]{Steinberg2010}}]\label{prop:weaksemilattice} 
An inverse semigroup $S$ is a weak semilattice if and only if for any partial action $\theta=\left(\{X_s\}_{s\in S},\{\theta_s\}_{s\in S}\right)$ of $S$ on a Hausdorff space $X$ such that $X_s$ is clopen for all $s \in S$, the groupoid of germs $S\ltimes X$ is Hausdorff.

In particular, if $S$ is a weak semilattice and $X$ is zero-dimensional, then the groupoid of germs $S\ltimes X$ is an ample Hausdorff groupoid.
\end{proposition}

\begin{remark}
The hypothesis that the domains of the partial action are clopen is necessary. For example, if $\G$ is a non-Hausdorff ample groupoid, then $\Gop$ is still a semilattice, however, as in Example~\ref{ex:amplegermesopenbisections}, the groupoid of germs $\Gop\ltimes\Go\cong\G$ is not Hausdorff.
\end{remark}

\section{Partial actions from associated groups and inverse semigroups}\label{sectionassociatedactionsandisomorphicgroupoidofgerms}

We will now describe how to construct partial actions of groups from actions of inverse semigroups and vice-versa. The class of inverse semigroups which allows us to do this in a more precise manner is that of \emph{$E$-unitary inverse semigroups}.

To each inverse semigroup $S$ we can naturally associate a group $\mathbf{G}(S)$: define a relation in $S$ by
\[\ntag\label{relation}
s\sim t\iff\quad\text{there exists}\quad u\in S \quad\text{such that}\quad u\leq s,t.
\]

Alternatively, $s\sim t$ if and only if there exists $e\in E(S)$ such that $se=te$. From this and the fact that the order of $S$ is preserved under products and inverses, it follows that $\sim$ is in fact a congruence, so we endow $S/\!\!\sim$ with the quotient semigroup structure. Given $s\in S$, we denote by $[s]$ the equivalence class of $s$ with respect to the relation (\ref{relation}). The following proposition is a particular case of Proposition \ref{prop:quotientoftrnsformationsemigroupoidisagroupoid} and Theorem \ref{theo:universalpropertyofgroupoidofgerms} (see Example \ref{example:maximalgroupimageastransformationgroupoid}).

\begin{proposition}[{\cite[Proposition~2.1.2]{Paterson1999}}]
Let $S$ an inverse semigroup. The quotient 
\[\mathbf{G}(S):=S/\!\!\sim\]
is a group. Furthermore, $\mathbf{G}(S)$ is the maximal group homomorphic image of $S$ in the sense that if $\psi\colon S \rightarrow G$ is a homomorphism and $G$ is a group, then $\psi$ factors through $\mathbf{G}(S)$.
\end{proposition}

\begin{example}
If $G$ is a group then $\mathbf{G}(G)$ is isomorphic to $G$.
\end{example}

\begin{example}
If $L$ is a (meet-)semilattice then $\mathbf{G}(L)=\left\{1\right\}$ is the trivial group.
\end{example}

\begin{example}
If $S$ is an inverse semigroup with a zero, then $\mathbf{G}(S)=\left\{1\right\}$ is the trivial group.
\end{example}

Recall that an inverse semigroup $S$ is \emph{$E$-unitary} if whenever $e,s\in S$, $e\leq s$ and $e\in E(S)$, we have $s\in E(S)$ as well. We first reword the $E$-unitary property in terms of compatibility of elements. Two elements $s,t$ of an inverse semigroup $S$ are \emph{compatible} if $s^*t$ and $st^*$ are idempotents. In this case, the set $\left\{s,t\right\}$ has infimum $s\land t=\inf\left\{s,t\right\}=st^*t=ts^*s$.

\begin{lemma}[{\cite[Theorem~2.4.6]{MR1694900}}]\label{lem:eunitarycompatibility}
$S$ is $E$-unitary if and only if $s,t,u\in S$ and $u\leq s,t$ implies that $s$ and $t$ are compatible.
\end{lemma}

We will now be interested in relating partial actions of inverse semigroups and partial actions of their maximal group images. A version this theorem has been proven in \cite[Lemma 3.8]{MR1848088} when considering global actions of inverse semigroups. The next theorem is a specific instance of \cite[Lemma 2.2]{Mikola2017}, where the author in fact considers a \emph{strictly weaker} notion of partial action -- namely, condition \ref{def:partialhomomorphism}(iii) is not required. Note that this condition is trivial when considering partial actions of groups, and thus we may apply \cite[Lemma 2.2]{Mikola2017} without problems.

\begin{theorem}[{\cite[Remark 2.3]{Mikola2017}}]\label{theoremfactorpartialactioneunitary}
Let $\theta=\left(\left\{X_s\right\}_{s\in S},\left\{\theta\right\}_{s\in S}\right)$ be a partial action of an $E$-unitary inverse semigroup $S$ on a topological space $X$. Then there is a unique partial action $\widetilde{\theta}=\left(\left\{X_{\gamma}\right\}_{\gamma\in \mathbf{G}(S)},\left\{\widetilde{\theta}_\gamma\right\}_{\gamma\in \mathbf{G}(S)}\right)$ of $\mathbf{G}(S)$ on $X$ such that for all $s\in S$,
\begin{enumerate}
\item[(i)] $X_\gamma=\bigcup_{[s]=\gamma}X_s$ for all $\gamma\in \mathbf{G}(S)$;
\item[(ii)] $\widetilde{\theta}_{[s]}(x)=\theta(x)$ for all $(s,x)\in S*X$;
\end{enumerate}
(in other words, $\widetilde{\theta}_{\gamma}$ is the join of $\left\{\theta_s:[s]=\gamma\right\}$ in $\mathcal{I}(X)$, which is commonly denoted by $\bigvee_{[s]=\gamma}\theta_s$). 
\end{theorem}

\begin{remark}
If one allows degenerate partial actions, then item (i) implies that $\theta$ is non-degenerate if and only if $\widetilde{\theta}$ is non-degenerate.
\end{remark}

A version of the next theorem has been proven in \cite{MR3231226}, when considering the canonical action of $S$ on the spectrum of its idempotent set $E(S)$. We prove the result for general partial actions of inverse semigroups on arbitrary topological spaces.

\begin{proposition}\label{prop:eunitarygroupoidofgerms}
Let $\theta$ be a partial action of an $E$-unitary group $S$ on a space $X$ and $\widetilde{\theta}$ be the induced action on $\mathbf{G}(S)$. Then
\[S\ltimes_\theta X\cong \mathbf{G}(S)\ltimes_{\widetilde{\theta}} X\]
\end{proposition}
\begin{proof}
Consider the map $[s,x]\mapsto ([s],x)$, which is well-defined by the definitions of the relations involved (see Equations \eqref{eq:equivalencegroupoidgerms} and \eqref{relation}). It is clearly a homomorphism, and surjectivity follows since $\widetilde{\theta}_{\gamma}=\bigvee_{[s]=\gamma}\theta_s$. As for injectivity, suppose $([s],x)=([t],y)$, where $[s,x],[t,y]\in S\ltimes_\theta X$. Then $x=y$ and $[s]=[t]$, so $x\in X_{s^*}\cap X_{t^*}$. Hence $s$ and $t$ are compatible, which implies $s(s^*st^*t)=t(s^*st^*t)$ (as both products describe the meet $s\land t$). Since $x\in X_{s^*}\cap X_{t^*}\subseteq X_{s^*s}\cap X_{t^*t}\subseteq X_{s^*st^*t}$ we conclude that $[s,x]=[t,y]$.\qedhere
\end{proof}

The two previous propositions describe a strong relationship between partial actions of an $E$-unitary inverse semigroup and partial actions of the associated group. The other direction initially reads as follows: ``How to associate, to a group $G$, an inverse semigroup $S$ together with a map $G\to S$ such that every partial action of $G$ factors through a partial action of $S$?'' The obvious answer would be $S=G$, so instead we look for \emph{global actions} of our semigroup $S$. This is the content of the paper \cite{Exel1998}:

Given a group $G$, let $\mathbf{S}(G)$ be the universal semigroup generated by symbols of the form $[t]$, where $t\in G$, modulo the relations
\begin{enumerate}[label=(\roman*)]
\item $[s^{-1}][s][t]=[s^{-1}][st]$;
\item $[s][t][t^{-1}]=[st][t^{-1}]$;
\item $[s][1]=[s]$;
\item $[1][s]=[s]$;
\end{enumerate}

Exel proved that $\mathbf{S}(G)$ is an inverse semigroup with unit $[1]$ (see \cite[Theorem~3.4]{Exel1998}). We will describe all the necessary properties of $\mathbf{S}(G)$ that we will need. For every $g\in G$, the inverse of $[g]$ is $[g^{-1}]$. Let us denote 
\[\epsilon_g=[g][g^{-1}].\]
By \cite[Proposition~2.5~and~3.2]{Exel1998}, for each $\gamma\in\mathbf{S}(G)$, there is a unique $n\geq 0$ and distinct elements $r_1,\ldots,r_n, g\in G$ such that
\begin{enumerate}
\item $\gamma=\epsilon_{r_1}\cdots\epsilon_{r_n}[g]$, (if $n=0$, this is simply $[g]$), and
\item $r_i\neq 1$ for all $i$.
\end{enumerate}
We call such a decomposition $\gamma=\epsilon_{r_1}\cdots\epsilon_{r_n}[g]$ the \emph{standard form} of $\gamma$, which is unique up to the order of $r_1,\ldots,r_n$. Moreover, given $g,r\in G$, we have $[g]\epsilon_r=\epsilon_{gr}[g]$.
Thus, for $\gamma=\epsilon_{r_1}\cdots\epsilon_{r_n}[g] \in \mathbf{S}(G)$, the inverse of $\gamma$ is written in standard form as
\[\gamma^*=[g^{-1}]\epsilon_{r_n}\cdots\epsilon_{r_1}=\epsilon_{g^{-1}r_n}\cdots\epsilon_{g^{-1}r_1}[g^{-1}],\]
The idempotents of $\mathbf{S}(G)$ are the elements of the form $\epsilon=\epsilon_{r_1}\cdots\epsilon_{r_n}[1]$.

For any group $G$ the inverse semigroup associated $\mathbf{S}(G)$ is $E$-unitary (\cite[Remark~3.5]{Exel1998}). Indeed, suppose $\gamma \in \mathbf{S}(G)$, $\epsilon \in E(\mathbf{S}(G))$ and $\epsilon \leq \gamma$. Writing $\gamma$ and $\epsilon$ in standard form, we obtain
\[\gamma=\epsilon_{s_1}\cdots\epsilon_{s_n}[s]\qquad\text{and}\qquad\epsilon=\epsilon_{e_1}\cdots\epsilon_{e_m}[1].\]
Since $\epsilon=\epsilon\gamma$ and  $[1]$ is a unit of $\mathbf{S}(G)$, we obtain
\[\epsilon_{e_1}\cdots\epsilon_{e_m}[1]=\epsilon=\epsilon\gamma=\epsilon_{e_1}\cdots\epsilon_{e_m}\epsilon_{s_1}\cdots\epsilon_{s_n}[s].\]
From the uniqueness of the standard form of $\epsilon$ we conclude that $s=1$ and $\gamma$ is an idempotent.

The main result of \cite{Exel1998} is the following property of the semigroup $\mathbf{S}(G)$. Although it is proven in principle only for partial on discrete sets, the same proof applies in the topological setting.

\begin{proposition}[{\cite[Theorem 4.2.]{Exel1998}}]\label{thm:actionuniversal}
Let $\theta=\left(\left\{X_s\right\}_{s\in S},\left\{\theta_s\right\}_{s\in S}\right)$ be a topological partial action of a group $G$ on a space $X$. Then there is a unique topological action
$\overline{\theta}$ of $\mathbf{S}(G)$ on $X$ such that $\overline{\theta}_{[g]} = \theta_g$, for all $g \in  G$.
\end{proposition}

\begin{proposition}\label{prop:propriedadesdesg}
Let $G$ be a group and $\mathbf{S}(G)$ the universal semigroup of $G$. Then the map $G\to \mathbf{G}(\mathbf{S}(G))$, $g\mapsto [[g]]$, is an isomorphism.
\end{proposition}
\begin{proof}
First note that for all $s,t\in G$,
\[[s][t]=[s][t][t^{-1}[t]=[st]\epsilon_{t},\]
Thus the map $G\to \mathbf{S}(G)$, $g\mapsto [g]$, is a partial homomorphism, and the map $\mathbf{S}(G)\to \mathbf{G}(\mathbf{S}(G))$, $\alpha\mapsto[\alpha]$, is a homomorphism. So $g\mapsto[[g]]$ is a partial homomorphism between groups, hence a homomorphism.

Given $\alpha\in \mathbf{S}(G)$, since $\alpha=\epsilon_{s_1}\cdots\epsilon_{s_n}[s]$ for certain $s,s_1,\ldots,s_n$ we get $[\alpha]=[[s]]$, so $g\mapsto [[g]]$ is surjective.

If $[[g]]=1=[[1]]$, then there is an idempotent $\epsilon=\epsilon_{e_1}\cdots\epsilon_{e_n}[1]$ for which
\[\epsilon_{e_1}\cdots\epsilon_{e_n}[g]=\epsilon[g]=\epsilon[1]=\epsilon_{e_1}\cdots\epsilon_{e_n}\]
and the uniqueness of the standard form implies $g=1$.\qedhere
\end{proof}

\begin{corollary}\label{corollarygroupoidofgermsofgisthesameastheoneofsg}
Let $\theta$ be a partial action of an group $G$ on a space $X$ and $\widetilde{\theta}$ be the induced action of $\mathbf{S}(G)$. Then
\[G\ltimes_\theta X \cong \mathbf{S}(G)\ltimes_{\widetilde{\theta}} X.\]
\end{corollary}
\begin{proof}
Let $\gamma=\widetilde{\widetilde{\theta}}$, the partial action of $\mathbf{G}(\mathbf{S}(G))$ induced by $\widetilde{\theta}$ as in Theorem \ref{theoremfactorpartialactioneunitary}. Let us prove that for all $g\in G$, $\theta_g=\gamma_{[[g]]}$. From this fact and Proposition \ref{prop:propriedadesdesg}, it follows easily that \[G\ltimes_\theta X\to\mathbf{G}(\mathbf{S}(G))\ltimes_\gamma X,\qquad (g,x)\mapsto ([[g]],x)\]
is a topological groupoid isomorphism. Proposition \ref{prop:eunitarygroupoidofgerms} provides the isomorphism $\mathbf{G}(\mathbf{S}(G))\ltimes_\gamma X\cong \mathbf{S}(G)\ltimes_{\widetilde{\theta}}X$, so we are done.

Let $g\in G$ be fixed. By definition, $\gamma_{[[g]]}$ is the supremum of $\left\{\widetilde{\theta}_{s}:s\sim [g]\right\}$. From the uniqueness of the standard form of each $s\in\mathbf{S}(G)$, it follows that $s\sim[g]$ if and only if $s\leq[g]$, and thus we conclude that $\gamma_{[[g]]}=\widetilde{\theta}_{[g]}$.\qedhere
\end{proof}

Note that the construction $S\mapsto\mathbf{G}(S)$ is functorial, from the category $\normalfont{\textsc{\textbf{Inv}}_{\mathrm{part}}}$ of inverse semigroups and partial homomorphisms, to the category $\normalfont{\textsc{\textbf{Grp}}}$ of groups and their homomorphisms: Given $\theta\colon S\to T$ a partial homomorphism, the map $s\mapsto [\theta(s)]$ is a partial homomorphism from $S$ to the group $\mathbf{G}(T)$, hence a homomorphism, and thus it factors through a homomorphism $\mathbf{G}(\theta)\colon\mathbf{G}(S)\to\mathbf{G}(T)$.

Similarly, we have a functor $\mathbf{S}$ from $\normalfont{\textsc{\textbf{Grp}}}$ to $\normalfont{\textsc{\textbf{Inv}}}$, the subcategory of  $\normalfont{\textsc{\textbf{Inv}}_{\mathrm{part}}}$ consisting of semigroup homomorphisms. It is not hard to see (following the proof of Proposition \ref{corollarygroupoidofgermsofgisthesameastheoneofsg}) that $\mathbf{G}\circ\mathbf{S}$ is naturally isomorphic to the identity of $\normalfont{\textsc{\textbf{Grp}}}$.

\textbf{Question:} Which semigroups are isomorphic to $\mathbf{S}(G)$ for some group $G$? (Note that, up to isomorphism, we need $G=\mathbf{G}(S)$.) One condition for such a semigroup is that it satisfies the ascending chain condition.

The following interesting corollary shows that for such semigroups one can always extend partial actions to actions:

\begin{corollary}
Let $G$ be a group, $S=\mathbf{S}(G)$ and $\theta$ a partial action of $S$ on a set $X$. Then there exists an action $\alpha$ of $S$ on $X$ such that $\theta_s\leq\alpha_s$ for all $s\in S$ and $S\ltimes_\theta X=S\ltimes_\alpha X$.
\end{corollary}
\begin{proof}
Let
\begin{enumerate}
\item $\gamma$ be the partial action of $\mathbf{G}(S)$ induced by $\theta$;
\item $\gamma'$ be the composition of $\gamma$ with the canonical isomorphism $G\to\mathbf{G}(\mathbf{S}(G))$, $g\mapsto [[g]]$;
\item $\alpha$ be the action of $S=\mathbf{S}(G)$ induced by $\gamma'$;
\end{enumerate}
Then for all $s\in S$,
\[\theta_{[s]}\leq\gamma_{[[s]]}=\gamma'_{s}=\alpha_{[s]}\]
and
\[S\ltimes_\theta X\cong\mathbf{G}(S)\ltimes_{\gamma} X\cong G\ltimes_{\gamma'}X\cong \mathbf{S}(G)\ltimes_\alpha X=S\ltimes_\alpha X.\qedhere\]
\end{proof}

\section{Dual partial actions and their crossed products}\label{sectiondualpartialaction}

In \cite[Theorem~3.2]{Beuter2018}, Beuter and Gonçalves showed that any Steinberg algebra of a transformation groupoid given by a partial action of a group, $A_R(G\ltimes X)$, is isomorphic to the crossed product $A_R(X)\rtimes G$. In the same paper (\cite[Theorem~5.2]{Beuter2018}), they proved that every Steinberg algebra associated with an ample Hausdorff groupoid $\G$, is isomorphic to the crossed product $A_R(\G[0])\rtimes \G^a$. Similarly, in \cite[Theorem~2.3.6]{Demeneghi2017}, Demeneghi proved that any Steinberg algebra of a groupoid of germs associated to an ample global action of an inverse semigroup is isomorphic to a crossed product $A_R(X)\rtimes S$, and as a consequence obtained the latter result presented by the previous authors (see \cite[Proposition~2.4.3]{Demeneghi2017}). However, \cite[Theorem 3.2]{Beuter2018} considers partial (non-global) actions of groups, and thus does not follow from \cite[Theorem 2.3.6]{Demeneghi2017}.

The objective of this section is to present a self-contained proof that generalizes both results above. More precisely, let $\theta=\left(\{X_s\}_{s\in S},\{\theta_s\}_{s\in S}\right)$ be a partial action of an inverse semigroup $S$ on a locally compact Hausdorff and zero-dimensional space $X$. Define, for each $s \in S$,
\[D_s=\{f \in A_R(X):\operatorname{supp}f\subseteq X_s \} \cong A_R(X_s),\]
where the rightmost isomorphism, $D_s\to A_R(X_s)$, is given by restriction: $f\mapsto f|_{X_s}$ (the inverse map extends elements of $A_R(X_s)$ as zero on $X\setminus X_s$). We then define
\[\begin{array}{rrcl}
 \alpha_s\colon& D_{s^*} &\to          &D_s\\
                 &f              & \mapsto &f \circ \theta_{s^*}
\end{array}\]
(or, more precisely, $\alpha_s(f)$ is the extensions of $f\circ\theta_{s^*}$ as zero on $X\setminus X_s$).

It is routine to check that  $\alpha=(\{D_s\}_{s\in S}, \{\alpha_s\}_{s \in S})$ is an algebraic partial action of $S$ on $A_R(X)$. In this case, we say that $\alpha$ is the \emph{dual partial action} of $\theta$.

We will now prove that the Steinberg algebra $A_R(S\ltimes_\theta X)$ is isomorphic to the crossed product $A_R(X)\rtimes_\alpha S$. To this, end, we will need a few technical lemmas.

\begin{lemma}\label{lemmapartitionbybasic}
Every compact-open bisection of $S\ltimes X$ is a disjoint union of compact-open elements of $\mathscr{B}_{\mathrm{germ}}$ (see Proposition \ref{propositionbgermisbasis}).
\end{lemma}
\begin{proof}
Let $A$ be a compact-open bisection of $S\ltimes X$. Since $\mathscr{B}_{\mathrm{germ}}$ is a basis for $S\ltimes X$ and $A$ is compact, there exists a finite family $\left\{[s_i,U_i]:1\leq i\leq n\right\}$ in $\mathscr{B}_{\mathrm{germ}}$ such that each $U_i$ is a compact-open subset of $X$ and $A=\bigcup_{i=1}^n [s_i,U_i]$.

Let $W_1=U_1$ and for $i\geq 2$, let $W_i=U_i\setminus\bigcup_{j=1}^{i-1}U_i$. Then the $W_i$ are all compact-open subsets of $X$, and
\[\so\left(\bigcup_{i=1}^n[s_i,W_i]\right)=\bigcup_{i=1}^n W_i=\bigcup_{i=1}^nU_i=\so\left(\bigcup_{i=1}^n[s_i,U_i]\right)=\so(A).\]
Since the source map is injective on $A$, we have $A=\bigcup_{i=1}^n[s_i,W_i]$. Moreover, if $i\neq j$, then
\[\so([s_i,W_i]\cap [s_j,W_j])\subseteq \so[s_i,W_i]\cap\so[s_j,W_j]=W_i\cap W_j=\varnothing,\]
and thus we have a partition of $A$ by compact-open elements of $\mathscr{B}_{\mathrm{germ}}$.\qedhere
\end{proof}

\begin{lemma}\label{lemmasamedomain}
For every pair of finite families $s_1,\ldots,s_n\in S$ and $r_1,\ldots,r_n\in R$, and for every compact-open subset $V\subseteq\bigcap_{i=1}^n X_{s_i^*}$, if $\sum_{i=1}^nr_i 1_{[s_i,V]}=0$ in $A_R(S\ltimes X)$, then $\sum_{i=1}^n\overline{r_i1_{\theta_{s_i}(V)}\delta_{s_i}}=0$ in $A_R(X)\rtimes S$.
\end{lemma}
\begin{proof}
We proceed by induction on $n$. The case $n=1$ is trivial, for $r_1 1_{[s_1,V]}=0$ implies that either $r_1=0$ or $V=\varnothing$, and in either case we have $\overline{r_11_{\theta_{s_1}(V)}\delta_{s_1}}=0$.

Assume then that the statement is valid for $n$, and that we have a sum with $n+1$ elements, of the form
\[\ntag\label{eqlemmasamedomaininduction}\sum_{i=1}^n r_i 1_{[s_i,V]}+r1_{[s,V]}=0\]

For every $x\in V$, consider the finite subcolletion $F(x)=\left\{i\in\left\{1,\ldots,n\right\}:[s,x]\in [s_i,V]\right\}$. Applying both sides of \eqref{eqlemmasamedomaininduction} on $[s,x]$, we obtain
\[r=-\sum_{i\in F(x)}r_i.\ntag\label{eq:rassumofri}\]
Moreover, by definition of $F(x)$, we have $[s,x]=[s_i,x]$ for all $i\in F(x)$, so there exists $t_x\in S$ such that
\begin{enumerate}
    \item $t_x\leq s,s_i$ ($i\in F(x)$);
    \item $x\in X_{t_x^*}$.
\end{enumerate}
Fix any compact-open neighbourhood $W_x\subseteq V\cap X_{t_x^*}$ of $x$.

The collection $\left\{W_x:x\in V\right\}$ is an open cover of $V$, so it admits a finite subcover: There exist $x_1,\ldots,x_M\in V$ such that $\left\{W_1,\ldots,W_M\right\}$ is a cover of $V$, where $W_j=W_{x_j}$. We may, with the same argument as in Lemma \ref{lemmapartitionbybasic}, assume that the $W_j$ are pairwise disjoint. Denote also $F_j=F(x_j)$ and $t_j=t_{x_j}$.

Given any $j$, we apply Equation \eqref{eq:rassumofri} on $x=x_j$ to obtain $r=-\sum_{i\in F_j}r_i$. Thus
\[r1_{[s,V]}=\sum_{j=1}^M r1_{[s,W_j]}=-\sum_{j=1}^M\sum_{i\in F_j}r_i1_{[s,W_j]}=-\sum_{j=1}^M\sum_{i\in F_j}r_i1_{[s_i,W_j]}\]
where we used $t_j\leq s,s_i$ ($i\in F_j$) and $W_j\subseteq X_{t_j^*}$ in the third equality. Then
\[\ntag\label{eqlemmasamedomain2}
    0=\sum_{i=1}^n r_i1_{[s_i,V]}+r1_{[s,V]}=\sum_{j=1}^M\left(\sum_{i=1}^n r_i1_{[s_i,W_j]}-\sum_{i\in F_j}r_i1_{[s_i,W_j]}\right)=\sum_{j=1}^M\left(\sum_{i\not\in F_j} r_i1_{[s_i,W_j]}\right).
\]
Now note that
\[\supp\left(\sum_{i\not\in F_j}r_i1_{[s_i,W_j]}\right)\subseteq\so^{-1}(W_j),\]
and these sets are pairwise disjoint since the $W_j$ are pairwise disjoint. Equation \eqref{eqlemmasamedomain2} then implies that for each $j$,
\[\sum_{i\not\in F_j}r_i1_{[s_i,W_j]}=0.\ntag\label{eqlemmasamedomain3}\]
Using the induction hypothesis on Equation \eqref{eqlemmasamedomain3} and summing over $j$ we obtain
\begin{align*}
0&=\sum_{j=1}^M\left(\sum_{i\not\in F_j}\overline{r_i1_{\theta_{s_i}(W_j)}\delta_{s_i}}\right)=\sum_{j=1}^M\left(\sum_{i=1}^n\overline{r_i1_{\theta_{s_i}(W_j)}\delta_{s_i}}-\sum_{i\in F_j}\overline{r_i 1_{\theta_{s_i}(W_j)}\delta_{s_i}}\right)\\
&=\sum_{i=1}^n \overline{r_i1_{\theta_{s_i}(V)}\delta_{s_i}}-\sum_{j=1}^M \sum_{i\in F_j} \overline{r_i1_{\theta_{s_i}(W_j)}\delta_{s_i}}.\ntag\label{eqlemmasamedomain4}
\end{align*}
Given $i\in F_j$, we have $t_j\leq s,s_i$ and $W_j\subseteq X_{t_j^*}$, so we have can again apply Equation \eqref{eq:rassumofri}, and the fact that $\left\{W_j:j=1,\ldots,M\right\}$ is a partition of $V$ to obtain
\[\sum_{j=1}^M\sum_{i\in F_j}\overline{r_i 1_{\theta_{s_i}(W_j)}\delta_{s_i}}=\sum_{j=1}^M\sum_{i\in F_j}\overline{r_i1_{\theta_{s}(W_j)}\delta_s}=\sum_{j=1}^M-r\overline{1_{\theta_s(W_j)}\delta_s}=-\overline{r1_{\theta_s(V)}\delta_s},\]
so the induction step follows from \eqref{eqlemmasamedomain4}.\qedhere
\end{proof}

\begin{lemma}\label{lemmafunctionfromsteinbergalgebraiswelldefined}
If $\sum_{i=1}^nr_i 1_{[s_i,U_i]}=0$ in $A_R(S\ltimes X)$, then $\sum_{i=1}^n\overline{r_i 1_{\theta_{s_i}(U_i)}\delta_{s_i}}=0$ in $A_R(X)\rtimes S$.
\end{lemma}
\begin{proof}
All the subset $U_1,\ldots,U_n$ are compact-open, and thus so is $U=\bigcup_{i=1}^n U_i$. We may find\footnote{This is a combinatorial fact easily proven by induction, or with the following argument: For any of the $2^n-1$ non-zero sequences $S\in \left\{0,1\right\}^n\setminus\left\{0\right\}$, we set $V_{S}=\bigcap_{i\in S^{-1}(1)}U_i$, and disregard any of these sets which are empty.} a partition $\left\{V_1,\ldots,V_m\right\}$ (where $m\leq 2^n-1$) of $U$ by compact-open subsets of $X$ with the property that each $U_i$ is the union of some of the sets $V_j$. In this case, $U_i\cap V_j\neq\varnothing$ if and only if $V_j\subseteq U_i$. We then have
\[\ntag\label{eqlemmafunctionfromsteinbergalgebraiswelldefined}
    0=\sum_{i=1}^n r_i1_{[s_i,U_i]}=\sum_{i=1}^n\left(\sum_{j:V_j\subseteq U_i}r_i1_{[s_i,V_j]}\right)=\sum_{j=1}^m\left(\sum_{i:V_j\subseteq U_i}r_i1_{[s_i,V_j]}\right).
\]
For each $j$, we have $\supp\left(\sum_{i:V_j\subseteq U_i}r_i1_{[s_i,V_j]}\right)\subseteq\so^{-1}(V_j)$, and these sets are pairwise disjoint. Equation \eqref{eqlemmafunctionfromsteinbergalgebraiswelldefined} implies that for each $j$, $\sum_{i:V_j\subseteq U_i}r_i1_{[s_i,V_j]}=0$. By Lemma \ref{lemmasamedomain}, $\sum_{i:V_j\subseteq U_i}\overline{r_i1_{\theta_{s_i}(V_j)}\delta_{s_i}}=0$ for each $j$. Summing over $j$,
\[0=\sum_{j=1}^m\left(\sum_{i:V_j\subseteq U_i}\overline{r_i1_{\theta_{s_i}(V_j)}\delta_{s_i}}\right)=\sum_{i=1}^n\left(\sum_{j: V_j\subseteq U_i}\overline{r_i 1_{\theta_{s_i}(V_j)}\delta_{s_i}}\right)=\sum_{i=1}^n \overline{r_i1_{\theta_{s_i}(U_i)}\delta_{s_i}}.\qedhere\]
\end{proof}

\begin{theorem}\label{theo:steinbergiscrossed}
Let $\theta=\left(\{X_s\}_{s\in S},\{\theta_s\}_{s\in S}\right)$ be a partial action of an inverse semigroup $S$ on a locally compact Hausdorff and zero-dimensional topological space $X$. Then the Steinberg algebra of $S \ltimes_{\theta} X$ is isomorphic to the crossed product $A_R(X) \rtimes_{\alpha} S$, where $\alpha=(\{D_s\}_{s \in S}, \{\alpha_s\}_{s \in S})$ is the dual partial action of $\theta$.
\end{theorem}

\begin{proof}
We will use the notation introduced in the definition of crossed product, Definition~\ref{def:partialskewinversesemigroupring}.

We will first show the existence of a homomorphism $\phi$ of $\mathscr{L}$ to $A_R(S\ltimes X)$ that vanishes on the ideal $\mathscr{N}$, and thus factors through a homomorphism $\Phi$ of the quotient $\mathscr{L}/\mathscr{N}= A_R(X) \rtimes_{\alpha} S$.

Define $\phi\colon \mathscr{L} \rightarrow A_R(S\ltimes X)$ on a generating element $f_s\delta_s$ of $\mathscr{L}$ by 
\[
  \phi(f_s\delta_s)(a)=\begin{cases}
                          f_s(\ra(a)), & \text{if } a \in [s,X_{s^*}]\\
                                    0, &  \text{otherwise},
                        \end{cases}
\]
and extend $\phi$ linearly to all of $\mathscr{L}$.

We first need check that $\phi$ is well-defined, that is, $\phi(f_s\delta_s)$ is a linear combination of characteristic functions of bisections of $S\ltimes X$.

We first write $f_s=\sum_{i=1}^n r_i1_{\theta_s(U_i)}$ for certain $r_1,\ldots,r_n\in R$ and $U_i\subseteq X_{s^*}$ compact-open. Then
\[\ntag\label{equationniceformulaforphi}
\phi(f_s\delta_s)(a)=\begin{cases}\displaystyle{\sum_{i:a\in[s,X_{s^*}]\cap\ra^{-1}(\theta_s(U_i))}r_i},&\text{if }a\in[s,X_{s^*}]\\0,&\text{otherwise.}\end{cases}\]
As $[s,X_{s^*}]\cap\ra^{-1}(\theta_s(U_i))=[s,U_i]$, equation \eqref{equationniceformulaforphi} simply means that
\[\ntag\label{equsefuldefinitionofPhi}
\phi(f_s\delta_s)=\sum_{i=1}^nr_i1_{[s,U_i]}\qquad\text{whenever }f_s=\sum_{i=1}^n r_i1_{\theta_s(U_i)}.\]

Therefore $\phi$ is a well-defined $R$-module homomorphism from $\mathscr{L}$ to $A_R(S\ltimes X)$. Now, we will show that $\phi$ is multiplicative. By linearity of $\phi$, it is enough to verify that it is multiplicative on the generators. Notice that $\supp(\phi(f_s\delta_s))=[s,\theta_s^{-1}(\supp(f_s))]$ for every generator $f_s\delta_s$.

Let $f_s\delta_s, f_t\delta_t \in \mathscr{L}$ and $a \in S \ltimes X$. There are two possibilities:

\textbf{Case 1}: $a\not\in [s,X_{s^*}][t,X_{t^*}]=[st,\theta_t^{-1}(X_t\cap X_{s^*})]$.

Since $\supp(\phi(f_s\delta_s)\ast\phi(f_t\delta_t))\subseteq[s,X_{s^*}][t,X_{t^*}]$, then
\[\phi(f_s\delta_s)*\phi(f_t\delta_t)(a)=0\]
On the other hand, $(f_s\delta_s)(f_t\delta_t)=\alpha_s(\alpha_{s^*}(f_s)f_t)\delta_{st}$. Since
\[\supp(\alpha_{s^*}(f_s)f_t)=\theta_s^{-1}(\supp f_s)\cap\supp(f_t)\]
then
\[\supp(\alpha_s(\alpha_{s^*}(f_s)f_t))=\supp(f_s)\cap\theta_s(\supp(f_t)\cap X_{s^*})\]
and this set is contained in $X_s\cap\theta_s(X_t \cap X_{s^*})$, which is the domain of the composition $\theta_t^{-1}\circ\theta_s^{-1}=\theta_{t^*}\circ\theta_{s^*}$, and thus in the domain $X_{st}$ of $\theta_{(st)^*}$. Hence
\[\theta_{st}^{-1}(\supp(\alpha_s(\alpha_{s^*}(f_s)f_t)))=\theta_t^{-1}(\theta_s^{-1}(\supp f_s)\cap\supp f_t),\]
and therefore
\[\supp(\phi((f_s\delta_s)(f_t\delta_t)))=[st,\theta_t^{-1}(\theta_s^{-1}(\supp f_s)\cap\supp f_t)]\]
which is contained in $[st,\theta_t^{-1}(X_t\cap X_{s^*})]=[s,X_{s^*}][t,X_{t^*}]$, so
\[\phi((f_s\delta_s)(f_t\delta_t))(a)=0=(\phi(f_s\delta_s)\ast \phi(f_t\delta_t))(a)\]
as we expected.

\textbf{Case 2}: $a\in[s,X_{s^*}][t,X_{t^*}]$.

In this case, we can write $a=[s,x][t,y]$ for unique $x\in X_{s^*}$ and $y\in X_{t^*}$ with $\theta_t(y)=x$. Since $\supp(\phi(f_s\delta_s))\subseteq[s,X_{s^*}]$ then
\begin{align*}
(\phi(f_s\delta_s)\ast\phi(f_t\delta_t))(a)
	&=\sum_{b\in\ra^{-1}(\ra(a))}\phi(f_s\delta_s)(b)\phi(f_t\delta_t)(b^{-1}a)\nonumber\\
	&=\phi(f_s\delta_s)[s,x]\phi(f_t\delta_t)[t,y]\nonumber\\
	&=f_s(\theta_s(x))f_t(\theta_t(y)).
\end{align*} 
On the other hand, $a\in[s,X_{s^*}][t,X_{t^*}]\subseteq[st,X_{(st)^*}]$, so
\begin{align*}
\phi((f_s\delta_s)\ast(f_t\delta_t))(a)&=\phi(\alpha_s(\alpha_{s^*}(f_s)f_t)\delta_{st})(a)\\
&=\alpha_s(\alpha_{s^*}(f_s)f_t)(\ra(a))=\alpha_s(\alpha_{s^*}(f_s)f_t)(\theta_s(x))\\
&=(\alpha_{s^*}(f_s)f_t)(x)=f_s(\theta_s(x))f_t(x)=f_s(\theta_s(x))f_t(\theta_t(y))\\\
&=(\phi(f_s\delta_s)\ast\phi(f_t\delta_t))(a)
\end{align*}
as we desired.

Now let us prove that $\phi$ vanishes on the ideal $\mathscr{N}$. Since $\phi$ is a homomorphism, it is enough to show that $\phi$ is zero in elements of the form $f\delta_s-f\delta_t$, where $s\leq t$ and $f\in D_s$, because these elements generate $\mathscr{N}$. Let $a\in S\ltimes X$. Then
\begin{itemize}\setlength\itemsep{1ex}
\item if $a \in [s, X_{s^*}]$ then $a \in [t,X_{t^*}]$, and 
 \[\phi(f\delta_s-f\delta_t)(a)=f(\ra(a))-f(\ra(a))=0;\]
\item if $a \in [t,X_{t^*}]\setminus [s,X_{s^*}]$ then $\ra(a) \notin X_s$, because $\ra$ is injective on $[t,X_{t^*}]$, and $f(\ra(a))=0$ because $f \in D_s$. Thus 
  \[\phi(f\delta_s-f\delta_t)(a)=0-f(\ra(a))=0;\]  
  \item if $a \notin [t,X_{t^*}]$ then $a \notin [s,X_{s^*}]$ as well, so
 \[\phi(f\delta_s-f\delta_t)(a)=0-0=0.\]
\end{itemize}

Therefore, $\phi$ factors through the quotient $\mathscr{L}/\mathscr{N}=A_R(X)\rtimes_\alpha S$ to a map $\Phi\colon A_R(X)\rtimes_\alpha S\to A_R(S\ltimes X)$ satisfying $\Phi(\overline{f\delta_s})=\phi(f\delta_s)$ whenever $f\in D{s^*}$.

In order to prove that $\Phi$ is bijective, we will show the existence of a map $\Psi\colon A_R(S\ltimes X)\rightarrow A_R(X)\rtimes_\alpha S$ which is in fact the inverse map of $\Phi$.

By Lemma \ref{lemmapartitionbybasic}, $A_R(X)$ is generated, as an $R$-module, by characteristic functions of compact-open basic bisections (those of the form $1_{[s,U]}$, where $U\subseteq X_s$ is compact-open).

By Lemma \ref{lemmapartitionbybasic}, every element $f\in A_R(S\ltimes X)$ may be written as $f=\sum_{i=1}^n r_i 1_{[s_i,U_i]} \in A_R(S\ltimes X)$, where $r_1,\ldots,r_n \in R$ and $[s_1,U_1],\ldots,[s_n,U_n]\in\mathscr{B}_{\mathrm{germ}}$. Define
\[
  \Psi(f)=\Psi\left(\sum_{i=1}^n r_i1_{[s_i, U_i]} \right):=\sum_{i=1}^n \overline{r_i1_{\theta_{s_i}(U_i)}\delta_{s_i}}.
\]
By Lemma \ref{lemmafunctionfromsteinbergalgebraiswelldefined}, $\Psi$ is well-defined, and clearly additive. To prove that $\Psi$ is a left inverse to $\Phi$, let $\overline{f_s\delta_s}\in A_R(X)\rtimes S$ (where $f_s\in D_s$). We already know (Equation \eqref{equsefuldefinitionofPhi}) that, by writing $f_s=\sum_{i=1}^n r_i1_{\theta_s(U_i)}$, we have
\[\Psi(\Phi(\overline{f_s\delta_s}))=\Psi\left(\sum_{i=1}^nr_i1_{[s,U_i]}\right)=\sum_{i=1}^n\overline{r_i1_{\theta_s(U_i)}\delta_s}=\overline{f_s\delta_s}.\]
Since the elements $\overline{f_s\delta_s}$ generate $A_R(X)\rtimes S$ as an additive group, we conclude that $\Psi\circ\Phi$ is the identity of $A_R(X)\rtimes S$. Similarly, the elements of the form $r1_{[s,U]}$ (where $s\in S$ and $U\subseteq X_{s^*}$ is compact-open) generate $A_R(S\ltimes X)$ as an additive group, by Lemma \ref{lemmapartitionbybasic}, and Equation \eqref{equsefuldefinitionofPhi} again implies
\[\Phi(\Psi(r1_{[s,U]}))=\Phi(\overline{r1_{\theta_s(U)}\delta_s})=r1_{[s,U]},\]
therefore $\Phi\circ\Psi$ is the identify of $A_R(S\ltimes X)$.\qedhere
\end{proof}

\begin{remark}\label{remarkdiagonalsubalgebras}
Note that the diagonal subalgebra $D_R(S\ltimes X)\cong A_R((S\ltimes X)^{(0)})$ of $A_R(S\ltimes X)$ coincides with $\operatorname{span}\left\{1_{[e,U]}:e\in E(S),\ U\subseteq X_e\right\}$, and so it is mapped, under the isomorphism of the previous theorem, to the diagonal subalgebra $\operatorname{span}\left\{\overline{1_U\delta_e}:e\in E(S),\ U\subseteq X_e\right\}$ of the crossed product $A_R(X)\rtimes S$.
\end{remark}

\begin{corollary}\label{cor:isosteinbergbisection}
Let $\G$ be an ample groupoid. Then the Steinberg algebra $A_R(\G)$ is isomorphic to the crossed product $A_R(\Go)\rtimes_{\mu}\G^{op}$ and $A_R(\Go)\rtimes_{\eta} \G^{a}$, where $\mu$ and $\eta$ are the dual actions of the canonical actions of $\G^{op}$ and $\G^a$ on $\Go$.
\end{corollary}

\begin{proof}
By Example~\ref{ex:amplegermesopenbisections}, $\G$ is isomorphic to the groupoids of germs $\G^{op}\ltimes\Go$ and $\G^{a}\ltimes\Go$, given by the respective canonical actions of $\Gop$ and $\Ga$ on $\Go$. The result follows from Theorem~\ref{theo:steinbergiscrossed}.
\end{proof}

It is interesting to note that the crossed products $A_R(\G[0])\rtimes \G^{op}$ and $A_R(\G[0])\rtimes \G^{a}$ arise from \emph{global actions}, and not simply partial action as in the previous theorem. Further, using Theorem~\ref{theo:steinbergiscrossed} and Corollary~\ref{cor:isosteinbergbisection} to a groupoid of germs of a partial action, we obtain
\[A_R(X)\rtimes_{\alpha}S\cong A_R(S\ltimes X)\cong A_R(X)\rtimes_{\eta}(S\ltimes X)^{a},\]
where $\eta$ is dual to the canonical action of $(S\ltimes X)^a$ on $(S\ltimes X)^{(0)}\cong X$.

\section{Recovering a topological partial action from a crossed product}

In the previous section we realized the Steinberg algebra of an ample groupoid of germs as a crossed product. In this section we will be interested in the opposite direction, that is, to determine which crossed products of the form $A_R(X)\rtimes_{\alpha}S$ can be realized as Steinberg algebras $A_R(S\ltimes_\theta S)$ in such a way that $\alpha$ is induced by $\theta$. The first problem we deal with is to find conditions which allow us to obtain a topological partial action $\theta$ of $S$ on $X$ from an algebraic action $\alpha$.

It is well-know that given a partial action $\theta=(\{X_g\}_{g\in G},\{\theta_{g}\}_{g\in G})$ of a group $G$ on a locally compact Hausdorff topological space $X$, there is an associated partial action $\alpha=(\{D_g\}_{g\in G},\{\alpha_{g}\}_{g\in G})$ of $G$ on the C*-algebra $C_0(X)$, and conversely, every partial action of a group $G$ on $C_0(X)$ comes from a partial action of $G$ on $X$. In \cite{Beuter2016}, a similar relation is shown at the purely algebraic level. More precisely, let $\mathbb{K}$ be a field and denote by $\mathcal{F}_0(X)$  the algebra of all functions $X \rightarrow \mathbb{K}$ with finite support, endowed with the pointwise operations. Then there is a one-to-one correspondence between the partial actions of a group $G$ on $X$ and the partial actions of $G$ on $\mathcal{F}_0(X)$. 

In this section, we will show that the same occurs with partial actions of inverse semigroups. Throughout this section, we will consider that:
\begin{itemize}
\item $X$ and $Y$ are locally compact Hausdorff and zero-dimensional topological spaces;
\item $S$ is an inverse semigroup;
\item $R$ is a commutative unital ring; and
\item $A_R(X)$ is the Steinberg algebra of $X$, i.e., the $R$-algebra formed by all locally constant, compactly supported, $R$-valued functions on $X$, with the pointwise operations.
\end{itemize} 

In order to find a biunivocal correspondence between partial actions $\theta=(\{X_s\}_{s\in S},\{\theta_{s}\}_{s \in X_s})$ of $S$ on $X$ and the dual partial actions  $\alpha=(\{D_s\}_{s\in S}, \{\alpha_s\}_{s \in S})$ of $S$ in $A_R(X)$, we will need a few preliminary results.

Recall that a ring $A$ is said to have \emph{local units} if, for every finite subset $F$ of $A$, there exists an idempotent $e \in A$ such that $r=er=re$ for each $r\in F$. Such an element $e$ will be referred to as a local unit for the set $F$. A commutative unital ring $R$ is said to be \emph{indecomposable} if its only idempotents are $0$ and $1$ (the trivial ones).

We will prove that, when $R$ is indecomposable, there is a bijection between ideals with local units of $A_R(X)$ and open subsets of $X$. On one hand, if $U$ is an open subset of $X$, then 
\begin{equation}\label{eq:idealofLcX}
  \mathbf{I}(U):=\left\lbrace f\in A_R(X) : \operatorname{supp}(f)\subseteq U\right\rbrace \cong A_R(U)
\end{equation} 
is an ideal of $A_R(X)$ with local units. Indeed, if $f_1, \ldots, f_n \in \mathbf{I}(U)$ then  the characteristic function $1_K$, where $K=\bigcup_{i=1}^n \supp(f_i)$, is a local unit for these functions. Moreover, $U$ is compact if and only if $\mathbf{I}(U)$ has an identity, namely, the characteristic function $1_U$ is its identity.

\begin{proposition}\label{prop:ideals}
Suppose that $R$ is an indecomposable commutative unital ring. Then the map $U\mapsto \mathbf{I}(U)$ is an order isomorphism between the lattices of open subsets of $X$ and of ideals with local units of $A_R(X)$. The inverse map is given by $I\mapsto \mathbf{U}(I):=\bigcup_{f\in I}\supp f$.
\end{proposition} 
 
\begin{proof}
Let $I\subseteq A_R(X)$ be an ideal with local units. Then the inclusion $I\subseteq\mathbf{I}(\mathbf{U}(I))$ follows from the definitions of $\mathbf{I}$ and $\mathbf{U}$. For the converse, suppose $f\in A_R(X)$ and $\operatorname{supp}f\subseteq\mathbf{U}(I)=\bigcup_{g\in I}\operatorname{supp}(g)$. By compactness of $\operatorname{supp}f$, there are $f_1,\ldots,f_n\in I$ with $\operatorname{supp}(f)\subseteq\bigcup_{i=1}^n\operatorname{supp}(f_i)$.

Let $e\in I$ be a local unit for $f_1,\ldots,f_n$. Since $e$ is idempotent and $R$ is indecomposable then $e=1_C$ for some clopen $C\subseteq X$, and since $e$ is a local unit for $f_1,\ldots,f_n$ this means that $\bigcup_{i=1}^n\operatorname{supp}f_i\subseteq C$. Therefore $\operatorname{supp}f\subseteq C$, and $f=f1_C=fe\in I$. This proves that $\mathbf{I}(\mathbf{U}(I))=I$.

For the converse, let $U\subseteq X$ be open, so that the inclusion $\mathbf{U}(\mathbf{I}(U))\subseteq U$ is also immediate from the definitions of $\mathbf{I}$ and $\mathbf{U}$. If $x\in U$, simply take any compact-open subset $V$ with $x\in V\subseteq U$, so $1_V\in \mathbf{I}(U)$ and
\[x\in\operatorname{supp}1_V\subseteq\mathbf{U}(\mathbf{I}(U)),\]
which proves that $U=\mathbf{U}(\mathbf{I}(U))$.\qedhere
\end{proof}

\begin{corollary}
Suppose that $R$ is an indecomposable commutative unital ring. Then there is an order-isomorphism between unital ideals of $A_R(X)$ and compact-open subsets of $X$. 
\end{corollary}

The following is a particular case of \cite[Theorem 3.42]{cordeiro2018}. We sketch its proof for the sake of completeness.

\begin{proposition}\label{prop:isomorp}
Let $R$ be an indecomposable commutative unital ring. Then $\Gamma\colon A_R(Y) \rightarrow A_R(X)$ is an $R$-algebra isomorphism if and only if there exists a (necessarily unique) homeomorphism $\varphi\colon X \rightarrow Y$ such that $\Gamma(f)=f\circ\varphi$ for all $f\in A_R(X)$.
\end{proposition}
\begin{proof}
Given a commutative ring $A$, denote by $\Omega(A)$ the set of all maximal ideals with local units of $A$.

By Proposition \ref{prop:ideals}, the map $X\ni x\mapsto \mathbf{I}(X\setminus\left\{x\right\})\in \Omega(A_R(X))$ is a bijection, and it is also a homeomorphism when we endow $\Omega(A_R(X))$ with the topology generated by all sets of the form\[[f]=\left\{I\in\Omega(A_R(X)):f\not\in I\right\}\qquad (f\in A_R(X)).\]

Repeating the same argument with $Y$ in place of $X$, and using the fact that $\Gamma$ preserves maximal ideals with local units, we obtain a homeomorphism $\varphi\colon X\cong\Omega(A_R(X))\to\Omega(A_R(Y))\cong Y$ such that $\supp(f)=\varphi(\supp(\Gamma(f))$ for all $f\in A_R(Y)$.

Let $x\in X$ be fixed, and choose any compact-open neighbourhood $U$ of $x$ and let $e=1_{\varphi(U)}\in A_R(Y)$. Then $\Gamma(e)^2=\Gamma(e^2)=\Gamma(e)$, so $\Gamma(e)$ only takes values $0$ and $1$ since $R$ is indecomposable. Moreover, $\varphi(U)=\supp(e)=\varphi(\supp\Gamma(e))$, so $\supp(\Gamma(e))=U$, and therefore $\Gamma(e)=1_U$.

Now given $f\in A_R(Y)$, fix $r=f(\varphi(x))$. We have $f(\varphi(x))=re(\varphi(x))$, thus
\[\varphi(x)\not\in\supp(f-re)=\varphi(\supp(\Gamma(f)-r\Gamma(e))).\]
Therefore $x\not\in\supp(\Gamma(f)-r\Gamma(e))$, so
\[\Gamma(f)(x)=r\Gamma(e)(x)=r=f(\varphi(x)).\qedhere\]
\end{proof}

In particular, from the proposition above, we may conclude that there is bijective anti-homomorphism between the group of all homeomorphism from $X$ to $Y$, and the group of all $R$-algebra isomorphisms from $A_R(Y)$ to $A_R(X)$, given by
\[T\colon \operatorname{Homeo}(X,Y)\to\operatorname{Iso}(A_R(Y),A_R(X)),\qquad \varphi\mapsto T_\varphi,\]
where $T_\varphi(f)=f\circ\varphi$ (compare this with \cite[Corollary 3.43]{cordeiro2018}).

\begin{proposition}\label{prop:actionalphatheta}
Suppose that $R$ is an indecomposable commutative unital ring. If $\alpha=(\{D_s\}_{s \in G}, \{\alpha_s\}_{s \in S})$ is a partial action of $S$ on the algebra $A_R(X)$ for which each ideal $D_s$ has local units, then there is a partial action $\theta=(\{X_s\}_{s \in S}, \{\theta_s\}_{s \in S})$ of $S$ on $X$ such that $\alpha$ is the dual partial action coming from $\theta$ (see Section~\ref{sectiondualpartialaction}).
\end{proposition}

\begin{proof}
Let $\alpha$ be a partial action of $S$ in $A_R(X)$ satisfying the hypotheses above. By Proposition \ref{prop:ideals}, for each $s\in S$ there is an open subset $X_s\subseteq X$ such that
\[
D_s=\mathbf{I}(X_s)= \left\lbrace f \in A_R(X) : \operatorname{supp}f\subseteq X_s\right\rbrace \cong A_R(X_s).
\]

By Proposition~\ref{prop:isomorp}, for each isomorphism $\alpha_s\colon  A_R(X_{s^*}) \rightarrow A_R(X_s)$, there is a unique homeomorphism $\theta_{s^*}\colon X_s \rightarrow X_{s^*}$ such that 
\[\alpha_s(f)=f\circ\theta_{s^*}\qquad\text{for all }f\in A_R(X_{s^*})\cong D_{s^*}.\]
So we simply let $\theta= (\{X_s\}_{s \in S}, \{\theta_s\}_{s \in S})$, and it is clear that, as long as $\theta$ is indeed a partial action, then $\alpha$ is the dual partial action of $\theta$.

To finish the proof we need to we show that $\theta $ is indeed a partial action. By its very definition, each $X_s$ is open in $X$ and $\theta_s\colon X_{s^*}\to X_s$ is a homeomorphism. Non-degeneracy of $\theta$ can be proven as follows:

Let $x \in X$ and $f \in A_R(X)$ such that $x \in \supp(f)$. Since $A_R(X)=\operatorname{span}\bigcup_{s\in S}D_s$, we can write $f$ as $f=\sum_{i=1}^n f_i$ for certain elements $s_i\in S$ and $f_i\in D_{s_i}\cong A_R(X_{s_i})$. In particular,
\[\operatorname{supp}f\subseteq\bigcup_{i=1}^n\operatorname{supp} f_i\subseteq\bigcup_{i=1}^n X_{s_i}\]
and so $x \in X_{s_i}$ for some $i$. This proves that $X=\bigcup_{s\in S}X_s$.

So it remains only to prove that $s\mapsto \theta_s$ is a partial homomorphism. Let us verify the conditions of Definition~\ref{def:partialhomomorphism}:
\begin{enumerate}
\item[(i)] Given $s\in S$, we need to prove that $\theta_{s^*}=(\theta_s)^*$ $\alpha_{s^*}\circ\alpha_s$ is the identity on $D_{s^*}\cong A_R(X_{s^*})$, however for all $f\in D_{s^*}\cong A_R(X_{s^*})$,
\[f\circ\operatorname{id}_{X_{s^*}}=\alpha_{s^*}\circ\alpha_s(f)=\alpha_{s^*}(f\circ\theta_{s^*})=f\circ(\theta_{s^*}\circ\theta_s)\]
so the uniqueness part of Proposition~\ref{prop:isomorp} implies that $\theta_{s^*}\circ\theta_s=\operatorname{id}_{X_{s^*}}$. Similarly, $\theta_s\circ\theta_{s^*}=\operatorname{id}_{X_s}$, thus $\theta_{s^*}=(\theta_s)^*$.
\item[(ii)] Let $s,t\in S$. We need to prove that $\theta_s\circ\theta_t\leq\theta_{st}$. On one hand, note that (under the usual identification $A_R(U)\cong\mathbf{I}(U)$),
\begin{align*}
f\in A_R(\theta_t^{-1}(X_t\cap X_{s^*}))&\iff\operatorname{supp}f\subseteq\theta_t^{-1}(X_t\cap X_{s^*})\\
&\iff\operatorname{supp}(f\circ\theta_{t^*})\subseteq X_t\cap X_{s^*}\\
&\iff\operatorname{supp}(\alpha_t(f))\subseteq X_t\cap X_{s^*}\\
&\iff\alpha_t(f)\in D_t\cap D_{s^*},
\end{align*}
that is, under the canonical identification, $A_R(\theta_t^{-1}(X_t\cap X_{s^*}))\cong \alpha_t^{-1}(D_t\cap D_{s^*})$. Since $\alpha$ is a partial action, we obtain
\[A_R(\theta_t^{-1}(X_t\cap X_{s^*}))\cong\alpha_t^{-1}(D_t\cap D_{s^*})\subseteq D_{(st)^*}\cong A_R(X_{(st)^*})\]
which implies $\theta_t^{-1}(X_t\cap X_{s^*})\subseteq X_{(st)^*}$. The map $\alpha_{(st)^*}\circ\alpha_s\circ\alpha_t$ coincides with the identity on $\alpha_t^{-1}(D_t\cap D_{s^*})$, however
\[\alpha_{(st)^*}(\alpha_s(\alpha_t(f)))=\alpha_{(st)^*}(\alpha_s(f\circ\theta_{t^*}))=\alpha_{(st)^*}(f\circ\theta_{t^*}\circ\theta_{s^*})=f\circ\theta_{t^*}\circ\theta_{s^*}\circ\theta_{st}\]
so again uniqueness in Proposition~\ref{prop:isomorp} implies that $\theta_{t^*}\circ\theta_{s^*}\circ\theta_{st}$ is the identity on $\theta_t^{-1}(X_t\cap X_s^*)$. We can conclude that $\theta_s\circ\theta_t\leq\theta_{st}$.
\item[(iii)] Suppose $s\leq t$ in $S$. Let us prove that $\theta_s\subseteq\theta_t$. We have
\[A_R(X_{s^*})\cong D_{s^*}\subseteq D_{t^*}\cong A_R(X_{t^*})\]
so $X_{s^*}\subseteq X_{t^*}$. The restriction of $\alpha_{t^*}$ to $D_s$ coincides with $\alpha_{s^*}$, so for all $f\in D_s\cong A_R(X_s)$,
\[f\circ(\theta_t|_{X_{s^*}})=(f\circ\theta_t)|_{X_{s^*}}=\alpha_{t^*}(f)|_{X_{s^*}}=\alpha_{s^*}(f)=f\circ\theta_s\]\
and again the uniqueness part in Proposition~\ref{prop:isomorp} implies that $\theta_t|_{X_{s^*}}=\theta_s$, so $\theta_s\leq\theta_t$.\qedhere
\end{enumerate}
\end{proof}

\begin{corollary}
Suppose that $S$ is an inverse semigroup, that $R$ is an indecomposable commutative unital ring, and that $\alpha=(\{D_s\}_{s\in S}, \{\alpha_s\}_{s \in S})$ is an algebraic partial action of $S$ on $A_R(X)$, where each ideal $D_s$ has local units. Then $A_R(X)\ltimes_{\alpha} S$ is isomorphic to a Steinberg algebra $A_R(S\ltimes_{\theta} X)$ such that $\alpha$ is dual to the topological partial action $\theta$.
\end{corollary}

\begin{proof}
By Proposition~{\ref{prop:actionalphatheta}}, $\alpha$ is dual to a topological partial action $\theta$ of $S$ on $X$, and Theorem~\ref{theo:steinbergiscrossed} implies that $A_R(S\ltimes_\theta X)\cong A_R(X)\ltimes_{\alpha} X$.
\end{proof}

\section{Topologically principal partial actions}

In this section our main goal is to introduce \emph{topologically principal} partial actions of inverse semigroups, which will be used later in our study of continuous orbit equivalence. We then use this notion to describe $E$-unitary inverse semigroups in terms of the existence of certain topologically principal partial actions.

Let $\G$ be a groupoid. The \emph{isotropy group} at a point $x\in\Go$ is 
\[\G_x^x=\left\{a\in\G:\so(a)=\ra(a)=x\right\}.\]
Note that $\G_x^x$ is a group with the operation inherited from $\G$. The \emph{isotropy subgroupoid} of a groupoid $\G$ is the subgroupoid
\[\operatorname{Iso}(\G)=\bigcup_{x\in\Go}\G_x^x=\left\{a\in\G:\so(a)=\ra(a)\right\}.\]
Since $\G[0]$ is an open subset of $\operatorname{Iso}(\G)$, then $\G[0]\subseteq\operatorname{int}(\operatorname{Iso}(\G))$. Following the nomenclature of \cite{Renault2008}, a topological groupoid $\G$ is \emph{effective} if the converse inclusion holds, i.e., if $\G[0]=\operatorname{int}(\operatorname{Iso}(\G))$.

A topological groupoid $\G$ is \emph{topologically principal} if the set of points in $\Go$ with trivial isotropy group is dense in $\Go$. By \cite[Proposition~3.6]{Renault2008}, every Hausdorff topologically principal étale groupoid is effective (the Hausdorff property is necessary, as the groupoid constructed in Example \ref{nonhausdorffaction} is topologically principal but not effective). Conversely, if $\G$ is a second-countable effective (possibly non-Hausdorff) groupoid and $\G[0]$ satisfies the Baire property, then $\G$ is topologically principal.

The class of (global) actions of inverse semigroups which correspond to effective groupoids of germs was defined in \cite{Exel2016}. However, we will be interested in partial actions which correspond to topologically principal groupoids of germs. Since we will not make assumptions of second-countability or the Hausdorff property, it is important to distinguish effectiveness and topological principality of groupoids. 

\begin{remark}
    The nomenclature ``essentially principal'' has been used to mean either effective or topologically principal (or even slight variations) in different works. See \cite[Definition II.4.3]{MR584266}, \cite[Section 2.2]{MR3548134} and \cite[Definition 4.6(4)]{Exel2016}. To avoid confusion on this part, we settle with the nomenclature of \cite{Renault2008}.
    
    Moreover, distinct notions of topological freeness -- for either partial actions of countable groups or global actions of inverse semigroups (see \cite{Li2017} and \cite{Exel2016}) -- have natural generalizations to the context of partial actions of inverse semigroups, however they do not coincide in general.
    
    To avoid any confusion, partial actions which correspond to topologically principal or effective groupoids of germs will be called topologically principal or effective, respectively (so the term ``topologically free'' will not be used).
\end{remark}

Throughout this section, $\theta=\left(\left\{X_s\right\}_{s\in S},\left\{\theta_s\right\}_{s\in S}\right)$ will always denote a topological partial action of an inverse semigroup $S$ on a topological space $X$.

\begin{definition}[{\cite[Definition 4.1]{Exel2016}}]
Let $x\in X$ and $s\in S$. We say that
\begin{enumerate}[label=(\arabic*)]
    \item $x$ is \emph{fixed} by $s$ if $\theta_s(x)=x$;
    \item $x$ is \emph{trivially fixed} by $s$ if there exists $e\in E(S)$ such that $e\leq s$ and $x\in X_e$. (In particular, $x$ is fixed by $s$.)
\end{enumerate}
The partial action $\theta$ is \emph{effective} if for all $s\in S$, the interior of the set of fixed points of $s$ consists of the trivially fixed points of $s$, i.e.,
\[\operatorname{int}\left\{x\in X_{s^*}:\theta_s(x)=x\right\}=\bigcup\left\{X_e:e\in E(S) \text{ and } e\leq s\right\}.\]
\end{definition}

A proof analogous to that of \cite[Theorem~4.7]{Exel2016} proves that effective partial actions correspond to effective groupoids of germs.

\begin{proposition}\label{prop:effectiveactiongroupoid}
The groupoid of germs $S\ltimes_{\theta} X$ is effective if and only if $\theta$ is effective. 
\end{proposition}

If $\theta$ is a partial action of an inverse semigroup $S$ on a set $X$, the subset $\left\{s\in S:x\in X_{s^*}\right\}$ of $S$ will be denoted by $S_x$. 
\begin{definition}\label{def:topprincipalpaction}
We denote by $\Lambda(\theta)$ the set of points of $X$ which are trivially fixed whenever they are fixed, i.e.,
\[\Lambda(\theta)=\left\{x\in X:\text{for all } s\in S_x\text{, if }\theta_s(x)=x\text{ then there exists }e\in E(S)\cap S_x\text{ with }e\leq s\right\}.\]

We say that $\theta$ is \emph{topologically principal} if $\Lambda(\theta)$ is dense in $X$.
\end{definition}

Similarly to the two descriptions of the germ relation as in Equations \eqref{eq:equivalencegroupoidgerms} and \eqref{eq:equivalencegroupoidgerms2}, we can alternatively describe $\Lambda(\theta)$ as
\[\Lambda(\theta)=\left\{x\in X:\text{for all }s,t\in S_x\text{, if }\theta_s(x)=\theta_t(x)\text{ then there exists }u\in S_x\text{ with }u\leq s,t\right\}.\ntag\label{eq:betterdefinitionoflambdatheta}\]

Suppose now that $\theta$ is a partial action of $S$ on a \emph{discrete} space $X$ - that is, a set. As closures and interiors of discrete spaces are trivial, we may rewrite both topological principality and effectiveness of $\theta$ as follows: for all $(s,x)\in S\ast X$, if $\theta_s(x)=x$ then there exists $e\in S_x\cap E(S)$ with $e\leq s$. In particular, $\theta$ is effective if and only if it is topologically principal, thus we can unambiguously call it \emph{free}.

More generally, by a \emph{free} partial action $\theta$ of $S$ on a topological space $X$, we mean a partial action which is free when $X$ is regarded simply as a set. Equivalently, this is to say that $\Lambda(\theta)=X$.

In the case that $G$ is a group, a partial action $\theta$ of $G$ is free if for all $x\in X$ (and for all $g\in G_x$), one has that $\theta_g(x)=x$ implies $g=1$, where $1$ is the identity of $G$, which is the usual notion of freeness for partial group actions.

It is interesting to note that freeness of a topological partial action implies that the associated groupoid of germs is Hausdorff. However, this is not true for topologically principal partial actions.

\begin{proposition}
If the action $\theta$ is free, then the groupoid of germs $S\ltimes X$ is Hausdorff.
\end{proposition}
\begin{proof}
Suppose $[s,x]\neq [t,y]$. First assume that $\so[s,x]\neq\so[t,y]$, that is, $x\neq y$. As $X$ is Hausdorff, choose disjoint neighbourhoods $U$ and $V$ of $x$ and $y$ in $X$, respectively. Then $\so^{-1}(U)$ and $\so^{-1}(V)$ are disjoint neighbourhoods of $[s, x]$ and $[t,y]$, respectively. Similarly, if $\ra[s,x]\neq\ra[t,y]$, we may find disjoint neighbourhoods of $[s,x]$ and $[t,y]$, respectively.

We are done if we prove that the two cases above are the only possibilities. Suppose then $\so[s,x]=\so[t,y]$ and $\ra[s,x]=\ra[t,y]$, that is, $x=y$ and $\theta_s(x)=\theta_t(y)=\theta_t(x)$. By freeness of $\theta$, there is $u \in S$, such that $u\leq s,t$ and $x \in X_{u^*}$, which is equivalent to stating $[s,x]=[t,x]=[t,y]$, a contradiction.\qedhere
\end{proof}

\begin{example}
As in Example~\ref{nonhausdorffaction}, let $S=\mathbb{N}\cup\left\{\infty,z\right\}$ and $\theta$ be the Munn representation of $S$ on $X=E(S)=\mathbb{N}\cup\left\{\infty\right\}$, endowed with the same topology as the one-point compactification of $\mathbb{N}$. This is a topologically principal partial action, since $\Lambda(\theta)=\mathbb{N}$ is dense in $X$, however the associated groupoid of germs $S\ltimes X$ is not Hausdorff.
\end{example}

In the specific setting of topological partial actions of countable groups on locally compact Hausdorff and second-countable spaces, \cite{Li2017} adopts a notion of ``topological freeness'' which happens to coincide (in this specific setting) with both effectiveness and topologically principality partial actions (of groups). The following proposition can be proven as in \cite[Lemma~2.4]{Li2017}, as an application of Baire's Category Theorem.

\begin{proposition}
Suppose that $S$ is countable and that $X$ is locally compact Hausdorff. Then the partial action $\theta$ of $S$ on $X$ is topologically principal if and only if for all $s\in S$, the set
\[\left\{x\in X_{s^*}:\text{if }\theta_s(x)=x\text{ then there exists }e\in E(S)\cap S_x\text{ with }e\leq s\right\}\]
is dense in $X_{s^*}$.
\end{proposition}

We will now reword topological principality of a partial action in terms of the groupoid of germs $S\ltimes X$.

\begin{proposition}\label{prop:principalprincipal}
The groupoid of germs $S\ltimes X$ is topologically principal if and only if $\theta$ is topologically principal.
\end{proposition}

\begin{proof}
As usual, we may assume the action $\theta$ is non-degenerate and identify $X$ with $(S\ltimes X)^{(0)}$. Then it is enough to prove that, under this identification, $\Lambda(\theta)$ is the set of points of $X$ with trivial isotropy, i.e.,
\[\Lambda(\theta)=\left\{x\in X:(S \ltimes X)_x^x=\{x\}\right\}.\]
Let $x\in X$ be given. First suppose $x\in \Lambda(\theta)$ and $[s,x]\in(S\ltimes X)_x^x$. This means that $x=\ra[s,x]=\theta_s(x)$, so there is $e\in E(S)\cap S_x$, $e\leq s$, which implies $[s,x]=[e,x]\simeq x$.

Conversely suppose $(S\ltimes X)_x^x=\left\{x\right\}$ and let $s\in S_x$ with $\theta_s(x)=x$. This means that $[s,x]\in(S\ltimes X)_x^x$, and so $[s,x]\simeq x\simeq [e,x]$ for some idempotent $e\in S_x$. By definition of the groupoid of germs, we can find another idempotent $f\in S_x$ with $se=ef$, so in particular $ef$ is an idempotent, $ef\leq s$, and $x\in X_{ef}$. This proves $x\in \Lambda(\theta)$.
\end{proof}

We finish this section by describing how $E$-unitary inverse semigroups can be characterized in terms of their partial actions.

\begin{proposition}\label{theoremfactorpartialactioneunitaryfree}
Suppose that $S$ is $E$-unitary and that $\widetilde{\theta}=\left(\left\{X_{\gamma}\right\}_{\gamma\in \mathbf{G}(S)}, \left\{\widetilde{\theta}_\gamma\right\}_{\gamma\in \mathbf{G}(S)}\right)$ is the unique partial action of $\mathbf{G}(S)$ on $X$ given by Theorem~\ref{theoremfactorpartialactioneunitary}. Then $\theta$ is topologically principal if and only if $\widetilde{\theta}$ is topologically principal.
\end{proposition}
\begin{proof}
We will prove that $\Lambda(\theta)=\Lambda(\widetilde{\theta})$. Suppose that $x\in\Lambda(\theta)$, and that $s\in S$ is such that $x=\widetilde{\theta}_{[s]}(x)=\theta_s(x)$. As $x\in\Lambda(\theta)$, there exists $e\in E(S)\cap S_x$ with $e\leq s$. In particular, $se=e$, so $[s]=[e]=1$, the unit of $\mathbf{G}(S)$. This proves that $\Lambda(\theta)\subseteq\Lambda(\widetilde{\theta})$.

Conversely, assume $x\in\Lambda(\widetilde{\theta})$, and that $s\in S_x$ is such that $x=\theta_s(x)=\widetilde{\theta}_{[s]}(x)$. This implies that $[s]=1=[s^*s]$, so there is an idempotent $e\in E(S)$ with $se=s^*se$. In particular, $s\geq s^*se$, which is idempotent, so $s$ is itself an idempotent because $S$ is $E$-unitary. It follows that $s^*s\in E(S)\cap S_x$, and $s^*s=s$. This proves that $\Lambda(\widetilde{\theta})\subseteq\Lambda(\theta)$.\qedhere
\end{proof}

\begin{lemma}\label{lemmaesidempotents}
Suppose that $\theta$ is topologically principal, and that $X_{s}\neq \varnothing$ for all $s\in S$. Then $E(S)=\left\{s\in S:\theta_s\text{ is idempotent}\right\}$.
\end{lemma}
\begin{proof}
Suppose $\theta_s$ is an idempotent. Since $X_{s^*}\neq\varnothing$, choose any $x\in X_{s^*}\cap\Lambda(\theta)$. Then $\theta_s(x)=x$, which implies that there is some $e\in E(S)$ with $e\leq s$, so $s$ is idempotent because $S$ is $E$-unitary.\qedhere
\end{proof}

\begin{lemma}\label{lemmaconditionsonpartialactionwhichimplysiseunitary}
Let $S$ be an inverse semigroup and $\theta=\left(\left\{X_s\right\}_{s\in S},\left\{\theta_s\right\}_{s\in S}\right)$ be a partial action of $S$ a space $X$ such that
\begin{enumerate}[label=(\roman*)]
\item\label{lemmaconditionsonpartialactionwhichimplysiseunitary1} $\theta$ factors through $\mathbf{G}(S)$ -- there is a partial action $\widetilde{\theta}=\left(\left\{X_{[s]}\right\}_{[s]\in \mathbf{G}(S)},\left\{\widetilde\theta_{[s]}\right\}_{[s]\in \mathbf{G}(S)}\right)$ such that $\widetilde{\theta}_{[s]}(x)=\theta_s(x)$ for all $x\in X_{s^*}$;
\item\label{lemmaconditionsonpartialactionwhichimplysiseunitary2} $E(S)=\left\{s\in S:\theta_s\text{ is idempotent}\right\}$.
\end{enumerate}
Then $S$ is $E$-unitary.
\end{lemma}
\begin{proof}
Suppose $e\in E(S)$, $e\leq s$. We have $1=[e]=[s]$, thus for all $x\in X_{s^*}$, $\theta_s(x)=\widetilde{\theta}_{[s]}(x)=\widetilde{\theta}_1(x)=x$, so $\theta_s$ is an idempotent and $s$ is idempotent by \ref{lemmaconditionsonpartialactionwhichimplysiseunitary2}.\qedhere
\end{proof}

Given an inverse semigroup $S$, we will consider the \emph{canonical} action of $S$ on itself as the action $\alpha=\left(\left\{D_s\right\}_{s\in S},\left\{\alpha_s\right\}_{s\in S}\right)$, where $D_s=\left\{t\in S:t^*t\leq ss^*\right\}$, and $\alpha_s(t)=st$ for $t\in D_{s^*}$. (This action is usually considered when one proves the Vagner-Preston theorem.)

\begin{theorem}
$S$ is $E$-unitary if and only if it admits a topologically principal partial action satisfying \ref{lemmaconditionsonpartialactionwhichimplysiseunitary1} and \ref{lemmaconditionsonpartialactionwhichimplysiseunitary2} of Lemma \ref{lemmaconditionsonpartialactionwhichimplysiseunitary}.
\end{theorem}
\begin{proof}
One implication is proven in Lemma \ref{lemmaconditionsonpartialactionwhichimplysiseunitary}. Assume then that $S$ is $E$-unitary, and let us prove that the canonical action $\alpha$ of $S$ is free: Suppose $st=t$, where $tt^*\leq s^*s$. Then $s\geq stt^*=tt^*$, which is idempotent, so $s$ is itself an idempotent. This clearly implies that the action $\alpha$ is free. Condition \ref{lemmaconditionsonpartialactionwhichimplysiseunitary1} follows from Theorem~\ref{theoremfactorpartialactioneunitary}, and condition \ref{lemmaconditionsonpartialactionwhichimplysiseunitary2} from Lemma~\ref{lemmaesidempotents}.\qedhere
\end{proof}

In fact, condition \ref{lemmaconditionsonpartialactionwhichimplysiseunitary2} of Lemma \ref{lemmaconditionsonpartialactionwhichimplysiseunitary} is always satisfied by the canonical action $\alpha$ of an inverse semigroup $S$ on itself: if $\alpha_s$ is idempotent, then $s=ss^*s=\alpha_s(s^*s)=s^*s$ is idempotent. We thus obtain:
\begin{corollary}
$S$ is $E$-unitary if and only if the canonical action of $S$ factors through $\mathbf{G}(S)$.
\end{corollary}

\section{Continuous Orbit Equivalence}\label{sec:continuousorbitequivalence}

In \cite{Li2017}, Li characterized continuous orbit equivalence of topologically free partial actions of countable groups on second-countable, locally compact Hausdorff spaces in terms of diagonal-preserving isomorphisms of the associated C*-crossed products. In this section, we will extend the notion of continuous orbit equivalence to partial actions of inverse semigroups and characterize orbit equivalence of topologically principal systems in terms of diagonal-preserving isomorphisms of the associated crossed products.

Throughout this section, $\theta=\left(\left\{X_s\right\}_{s\in S},\left\{\theta_s\right\}_{s\in S}\right)$ and $\gamma=\left(\left\{Y_t\right\}_{t\in T},\left\{\gamma_t\right\}_{t\in T}\right)$ will always denote topological partial actions of inverse semigroup $S$ and $T$ on topological spaces $X$ and $Y$, respectively. Recall that $S\ast X=\left\{(s,x)\in S\times X:s\in S\text{ and }x\in X_{s^*}\right\}$ (and similarly for $T\ast Y$). We regard $S$ and $T$ as discrete topological spaces.

\begin{definition}\label{definitioncoe}
We say that $\theta$ and $\gamma$ are \emph{continuously orbit equivalent} if there exist a homeomorphism
\[\varphi\colon X \longrightarrow Y\]
and continuous maps
\[a\colon S\ast X \longrightarrow T\qquad\text{and}\qquad b\colon T\ast Y \longrightarrow S\]
such that for all $x\in X$, $s\in S_x$, $y\in Y$ and $t\in T_y$,
\begin{enumerate}[label=(\roman*)]
\item\label{definitioncoea} $ \varphi(\theta_s(x)) = \gamma_{a(s,x)}(\varphi(x))$;
\item\label{definitioncoeb} $ \varphi^{-1}(\gamma_t(y)) = \theta_{b(t,y)}(\varphi^{-1}(y))$.
\end{enumerate}
Implicitly, we require that $a(g,x) \in T_{\varphi(x)}$ and $b(t,y) \in S_{\varphi^{-1}(y)}$.
We call the triple $(\varphi,a,b)$ a \emph{continuous orbit equivalence} from $\theta$ to $\gamma$.
\end{definition}

Our next goal is to prove that continuous orbit equivalence of topologically principal partial actions is equivalent to the isomorphism of the respective groupoids of germs. For this, we need to prove some identities related to how the functions $a$ and $b$ above preserve the structure of $S$ and $T$.

\begin{lemma}\label{lem:seila}
Let $(\varphi,a,b)$ be a continuous orbit equivalence from $\theta$ to $\gamma$. Assume that $X$ and $Y$ are Hausdorff. Then the following implications hold:
\begin{enumerate}[label=(\alph*)]
    \item\label{lem:seilaa} $[s_1, x]=[s_2, x] \Rightarrow [a(s_1,x),\varphi(x)]=[a(s_2,x),\varphi(x)]$, for all $x\in X$ and $s_1, s_2 \in S_x$.
    \item\label{lem:seilab} $[a(s_1s_2, x),\varphi(x)]=[a(s_1, \theta_{s_2}(x))a(s_2, x),\varphi(x)]$ for all $x\in X$, $s_2\in S_x$ and $s_1\in S_{\theta_{s_2}(x)}$.
    \item\label{lem:seilac} $[b(a(s, x), \varphi(x)),x] = [s,x]$, for all $x\in X$ and $s\in S_x$.
\end{enumerate}

Analogous statements hold with $(\varphi^{-1}, b, a)$ in place of $(\varphi, a, b)$.
\end{lemma}

\begin{proof}
\begin{enumerate} \setlength\itemsep{1ex}
\item[(a)] Let $x\in X$ and $s_1, s_2 \in S_x$. Suppose that $[s_1,x]=[s_2,x]$. First, choose $s\leq s_1,s_2$ such that $x\in X_{s^*}$. Then choose an open neighbourhood $U\subseteq X_{s^*}$ of $x\in X$ such that
\[a(s_1, \tilde{x})=a(s_1, x)\qquad\text{and}\qquad a(s_2, \tilde{x})=a(s_2, x)\qquad\text{whenever }\tilde{x}\in U.\] Then for all $\tilde{x}\in U\cap\varphi^{-1}(\Lambda(\gamma))$ and for $i=1,2$, we have $[s_i,\tilde{x}]=[s,\tilde{x}]$, so
\begin{align*}
    \gamma_{a(s_i,x)}(\varphi(\tilde{x})&=\gamma_{a(s_i,\widetilde{x})}(\varphi(\tilde{x}))=\varphi(\theta_{s_i}(\tilde{x}))=\varphi(\ra[s_i,\tilde{x}])=\varphi(\ra[s,\tilde{x}]).
\end{align*}

It follows that $\gamma_{a(s_1,x)}(\varphi(\tilde{x}))=\gamma_{a(s_2,x)}(\varphi(\tilde{x}))$. As $\varphi(\tilde{x})\in\Lambda(\gamma)$, the description of $\Lambda(\gamma)$ as in Equation \eqref{eq:betterdefinitionoflambdatheta} implies that \[[a(s_1,x),\varphi(\tilde{x})]=[a(s_2,x),\varphi(\tilde{x})]\qquad\text{for all } x\in U\cap\varphi^{-1}(\Lambda(\gamma)).\ntag\label{eq:comparisonofasixvarphixtilde}\] In particular, $[a(s_i,x),\varphi(\tilde{x})]$ and $[a(s_i,x),\varphi(x)]$ belong to the bisection $[a(s_1,x),\varphi(U)]$, which is Hausdorff.

Since $\gamma$ is topologically principal, $\Lambda(\gamma)$ is dense in $Y$, so $U\cap \varphi^{-1}(\Lambda(\gamma))$ is dense in $U$ and therefore we may take the limit $\widetilde{x}\to x$ in Equation \eqref{eq:comparisonofasixvarphixtilde} and conclude that $[a(s_1,x),\varphi(x)]=[a(s_2,x),\varphi(x)]$,  limits are unique in Hausdorff spaces.

\item[(b)] Choose an open neighbourhood $U$ of $x \in X$ such that 
\[a(s_1s_2,\tilde{x})=a(s_1s_2,x),\quad a(s_1,\theta_{s_2}(\tilde{x}))= a(s_1,\theta_{s_2}(x))\quad\text{and}\quad a(s_2, \tilde{x}) = a(s_2,x)\]
for all $\tilde{x} \in U$. Then for all $\tilde{x} \in U \cap\varphi^{-1}(\Lambda(\gamma))$
\begin{align*}
  \gamma_{a(s_1s_2, \tilde{x})}(\varphi(\tilde{x})) 
     & = \varphi(\theta_{s_1s_2}(\tilde{x})) 
       = \varphi(\theta_{s_1}(\theta_{s_2}(\tilde{x})) 
       = \gamma_{a(s_1,\theta_{s_2}(\tilde{x}))}(\varphi(\theta_{s_2}(\tilde{x}))) \\
     & = \gamma_{a(s_1, \theta_{s_2}(\tilde{x}))}(\gamma_{a(s_2, \tilde{x})}(\varphi(\tilde{x})))
       = \gamma_{a(s_1, \theta_{s_2}(\tilde{x}))a(s_2, \tilde{x})}(\varphi(\tilde{x}))
\end{align*}
so, the same way as in item (a), the given property of $U$ and the definition of $\Lambda(\gamma)$ imply that $[a(s_1s_2,x),\varphi(\tilde{x})]=[a(s_1, \theta_{s_2}(x))a(s_2, x),\varphi(\tilde{x})]$. Since $\varphi^{-1}(\Lambda(\gamma))\cap U$ is dense in the Hausdorff space $U$, we conclude that $[a(s_1s_2,x),\varphi(x)]=[a(s_1,\theta_{s_2}(x))a(s_2,x),\varphi(x)]$ by taking the limit $\tilde{x}\to x$.

\item[(c)] Similarly to the previous items, take neighbourhoods $U$ of $x$ and $V$ of $\varphi(x)$ such that
\[a(s,\tilde{x})=a(s,x)\qquad\text{and}\qquad b(a(s,x),\tilde{y})=b(a(s,x),\varphi(x))\]
whenever $\tilde{x}\in U$ and $\tilde{y}\in V$. Then for all $\tilde{x}\in U\cap\varphi^{-1}(V)\cap \Lambda(\theta)$,
\[\theta_{b(a(s,\tilde{x}),\varphi(\tilde{x}))}(\tilde{x})=\varphi^{-1}(\gamma_{a(s,\tilde{x})}(\varphi(\tilde{x})))=\varphi^{-1}(\varphi(\theta_s(\tilde{x})))=\theta_s(x)\]
so the properties of $U$, $V$ and $\Lambda(\theta)$ yield $[b(a(s,x),\varphi(x)),\tilde{x}]=[s,\tilde{x}]$ and again taking $\tilde{x}\to x$ gives us the desired result.\qedhere
\end{enumerate}
\end{proof}

\begin{theorem}\label{theo:coegroupoidgermiso}
Suppose that $\theta$ and $\gamma$ are topologically principal, continuously orbit equivalent partial actions, and that $X$ and $Y$ are Hausdorff. Then $S \ltimes X$ and $T \ltimes Y$ are isomorphic as topological groupoids.
\end{theorem}
\begin{proof}
Let $(\varphi,a,b)$ be a continuous orbit equivalence from $\theta$ to $\gamma$ (as in Definition~\ref{definitioncoe}). Then the map
\begin{align*}
\Phi\colon S\ltimes X\to T\ltimes Y,\qquad \Phi[s,x]=[a(s,x),\varphi(x)]
\end{align*}
is a continuous groupoid homomorphism. Indeed, by Lemma~\ref{lem:seila}\ref{lem:seilaa}, $\Phi$ is well-defined, and item \ref{lem:seilab} of that lemma implies that $\Phi$ is a homomorphism. As $a$ and $\varphi$ are continuous, it follows that $\Phi$ is continuous. Similarly, the map  
 \begin{align*}
\Psi\colon T\ltimes Y\to S\ltimes X,\qquad\Psi[t,y]=[b(t,y),\varphi^{-1}(y)]
\end{align*}
is a continuous groupoid homomorphism, and $\Psi$ is a left inverse to $\Phi$ by Lemma~\ref{lem:seila}\ref{lem:seilac}. The same arguments with $(\varphi^{-1},b,a)$ in place of $(\varphi,a,b)$ prove that it is also a right inverse.\qedhere
\end{proof}

We will now be interested in constructing an orbit equivalence for two actions from an isomorphism of the corresponding groupoids of germs. Note that in general the continuous maps $a$ and $b$ in the definition of continuous orbit equivalence take values in discrete spaces (namely, the corresponding semigroups), and so $X$ and $Y$ are required to have sufficiently many clopen sets in order for a continuous orbit equivalence between the corresponding partial actions to exist. Thus we concentrate on spaces which have sufficiently many clopen sets and partial actions which respect this structure.

The required property for the topological spaces that we will need to consider is \emph{ultraparacompactness}, which is a stronger version of zero-dimensionality and covers most cases of interest (namely locally compact Hausdorff and zero-dimensional spaces which are also second countable or compact; see Example \ref{examplelindelofisultraparacompact}). We refer to \cite{MR0261565,MR1908882} and the references therein to finer properties, the history, and nontrivial examples of ultraparacompact spaces.

\begin{definition}
A Hausdorff topological space $X$ is \emph{ultraparacompact} if every open cover $\mathcal{U}$ of $X$ admits a refinement by clopen pairwise disjoint sets.
\end{definition}

Alternatively (see \cite[Proposition 1.2]{MR0261565}), a Hausdorff space $X$ is ultraparacompact if and only if it is paracompact\footnote{A topological space $X$ is \emph{paracompact} if every open cover of $X$ admits a locally finite refinement.}, and if whenever $F\subseteq O\subseteq X$, where $F$ is closed and $O$ is open, there is a clopen $C\subseteq X$ such that $F\subseteq C\subseteq O$.

\begin{example}\label{examplelindelofisultraparacompact}
Recall that a topological space $X$ is \emph{Lindelöf} if every open cover of $X$ admits a countable subcover. All compact spaces are Lindelöf, and all second-countable spaces are Lindelöf, and there are spaces which are compact but not second-countable and vice-versa.

Let us prove that every Lindelöf, Hausdorff and zero-dimensional space $X$ is ultraparacompact. Let $\mathscr{U}$ be an open cover of $X$. Since $X$ is zero-dimensional, there exists a refinement $\mathscr{V}$ of $\mathscr{U}$ by clopen sets, and we may assume that $\mathscr{V}$ is countable as $X$ is Lindelöf, say $\mathscr{V}=\left\{V_n:n\in\mathbb{N}\right\}$. Letting $V_0=\varnothing$, and defining $W_n=V_n\setminus\bigcup_{i=0}^{n-1}V_i$ for all $n\geq 1$, we obtain a refinement $\mathscr{W}=\left\{W_n:n\in\mathbb{N}\right\}$ of $\mathscr{U}$ by pairwise disjoint clopen sets.
\end{example}

\begin{definition}
    The topological partial action $\theta=(\left\{X_s\right\}_{s\in S},\left\{\theta_s\right\}_{s\in S})$ is \emph{almost ample} if $X$ is locally compact Hausdorff and all the subsets $X_s$ ($s\in S$) are ultraparacompact. (In particular, $X$ is zero-dimensional.)
\end{definition}

The class of almost ample partial actions is strictly larger class than the class of ``ample actions'' considered in \cite[Definition 5.2]{Steinberg2010}, as Example \ref{examplelindelofisultraparacompact} shows.

\begin{lemma}\label{lemmalocaldescriptionofcoe}
Suppose that the partial actions $\theta$ and $\gamma$ are almost ample. Let $\varphi\colon X\to Y$ be a continuous function. Then the following are equivalent:
 \begin{enumerate}[label=(\alph*)]
     \item\label{lemmalocaldescriptionofcoea} There exist a continuous function $a\colon S\ast X\to T$ such that for every $s\in S$ and $x\in X_{s^*}$, $\varphi(\theta_s(x))=\gamma_{a(s,x)}(\varphi(x))$;
     \item\label{lemmalocaldescriptionofcoeb} For every $s\in S$ and every $x\in X_{s^*}$, there exists a neighbourhood $U\subseteq X_{s^*}$ of $x$ and $t\in T$ such that $\varphi(\theta_s(\tilde{x}))=\gamma_t(\varphi(\tilde{x}))$ for all $\tilde{x}\in U$.
 \end{enumerate}
\end{lemma}
\begin{proof}
Assuming that \ref{lemmalocaldescriptionofcoea} is valid and given $(s,x)\in S\ast X$, we take $t=a(s,x)$ and $U=\left\{y\in X_{s^*}:a(s,y)=t\right\}$, which is open since $T$ is discrete and $a$ is continuous. Then the statement in \ref{lemmalocaldescriptionofcoeb} is valid.

Assume then that \ref{lemmalocaldescriptionofcoeb} is valid. Given $s\in S$, the condition in \ref{lemmalocaldescriptionofcoeb} and ultraparacompactness of $X_{s^*}$ allow us to find a clopen partition $\mathscr{U}_s$, and a family $\left\{t_U:U\in\mathscr{U}_s\right\}\subseteq T$ such that for all $U\in\mathscr{U}_s$ and all $x\in U$, $\varphi(\theta_s(x))=\gamma_{t_U}(\varphi(x))$.

We define $a\colon S\ast X\to T$, by setting $a(s,x)=t_U$, where $U$ is chosen as the unique element of $\mathscr{U}_s$ such that $x\in U$. Then \ref{lemmalocaldescriptionofcoea} holds.\qedhere
\end{proof}

We are now ready to prove that topological isomorphisms between Hausdorff groupoids of germs yield a continuous orbit equivalence between the respective partial actions.

\begin{theorem}\label{theo:groupoidgermscoe}
Suppose that $\theta$ and $\gamma$ are almost ample topological partial actions, and that the groupoids of germs $S \ltimes X$ and $T \ltimes Y$ are topologically isomorphic. Then $\theta$ and $\gamma$ are continuously orbit equivalent.
\end{theorem}

\begin{proof}
Let $\Phi\colon S \ltimes X \rightarrow  T \ltimes Y$ be an isomorphism of topological groupoids. As $(S\ltimes X)^{(0)}=X$ and $(T\ltimes Y)^{(0)}=Y$, the restriction
\[\varphi:= \Phi|_X\colon X \to Y\]
is a homeomorphism.

We will use Lemma~\ref{lemmalocaldescriptionofcoe}. Let $s\in S$ and $x\in X_{s^*}$ be fixed. Since $[s,X_{s^*}]$ is a neighbourhood of $[s,x]$, then $\Phi[s,X_{s^*}]$ is a neighbourhood of $\Phi[s,x]$, so we may choose a basic neighbourhood $[t,V]$ of $T\ltimes Y$, where $t\in T$ and $V\subseteq Y_ {t^*}$, such that $\Phi[s,x]\in [t,V]$. Consider the neighbourhood $U=\varphi^{-1}(V)\cap X_{s^*}$ of $x$.

If $a\in[t,V]$, then $\ra(a)=\gamma_t(\so(a))$. It follows that for all $\tilde{x}\in U$, we have
\[\varphi(\theta_s(\tilde{x}))=\Phi(\ra[s,\tilde{x}])=\ra(\Phi[s,\tilde{x}])=\gamma_t(\so(\Phi[s,\tilde{x}]))=\gamma_t(\Phi(\so[s,x]))=\gamma_t(\varphi(\tilde{x})).\]

Thus Lemma \ref{lemmalocaldescriptionofcoe}\ref{lemmalocaldescriptionofcoeb} holds, which implies item \ref{lemmalocaldescriptionofcoea} of the same lemma. This yields us the function $a\colon S\ast X\to T$ satisfying property \ref{definitioncoea} of Definition \ref{definitioncoe}. The function $b\colon T\ast Y\to S$ satisfying Definition \ref{definitioncoe}\ref{definitioncoeb} is constructed in a similar manner, and we therefore obtain a continuous orbit equivalence from $\theta$ to $\gamma$.\qedhere
\end{proof}

\begin{example}
In some sense, the hypothesis that the domains of the partial actions are ultraparacompact is the weakest condition possible one needs to assume to obtain Theorem \ref{theo:groupoidgermscoe}.

For example, suppose that $X$ is Hausdorff, but not ultraparacompact (for example, $X=\omega_1$, the first uncountable ordinal with the order topology, which is in fact locally compact and zero-dimensional).

Let $\mathscr{U}$ be any clopen cover of $X$ which does not admit any refinement by pairwise disjoint clopen sets. We let $S$ be the collection of all finite intersections of elements of $\mathscr{U}$, which is a semigroup (actually, a semilattice) under intersection, and let $\theta=\left(\{X_A\}_{A\in S},\{\theta_A\}_{A\in S}\right)$ be the natural action of $S$ on $X$: $X_A=A$ and $\theta_A=\operatorname{id}_A$, the identity of $A$, for all $A\in S$.

Also, let $G=\left\{1\right\}$ be the trivial group and $\gamma$ the trivial action of $G$ on $X$: $\gamma_1=\operatorname{id}_X$.

Then both $S\ltimes X$ and $G\ltimes X$ are isomorphic, as topological groupoids, to $X$. Let us prove, however, that $\theta$ and $\gamma$ are not continuously orbit equivalent. Suppose, on the contrary, that $(\varphi,a,b)$ were a continuous orbit equivalence from $\theta$ to $\gamma$. For all $x\in X$, we have
\[\varphi(x)=\varphi(\gamma_1(x))=\theta_{b(1,x)}(\varphi(x))\]
which in particular implies that \[\varphi(x)\in b(1,x).\ntag\label{exampleisomorphicbutnotcoenonultraparacompact}\]

For each $A\in S$, we consider the subset $U_A=\left\{x\in X:b(1,x)=A\right\}$ of $X$. The collection $\mathscr{V}:=\left\{\varphi(U_A):A\in S\right\}$ is a clopen partition of $X$, and Equation \eqref{exampleisomorphicbutnotcoenonultraparacompact} means that $\varphi(U_A)\subseteq A$ for each $A\in S$. Thus $\mathscr{V}$ is a refinement of $S$, and therefore a refinement of $\mathscr{U}$, contradicting the choice of $\mathscr{U}$.
\end{example}

\subsection*{Topological full pseudogroups}

We will use a similar terminology to that of \cite{MR2876963}. For each compact-open bisection $U\in\Ga$ of an ample groupoid $\mathcal{G}$, we denote by $\tau_U$ the homeomorphism given by the canonical action of $\Ga$ on $\Go$, namely $\tau_U=\ra\circ (\so|_U^{-1})\colon\so(U)\to\ra(U)$. Recall from Example~\ref{examplecanonicalaction} that $U\mapsto\tau_U$ is a homomorphism from $\Ga$ to $\mathcal{I}(\Go)$.

\begin{definition}
The \emph{topological full pseudogroup} of an ample groupoid is the semigroup 
\[[[\mathcal{G}]]=\left\{\tau_U:U\in\Ga\right\}\]
\end{definition}

\begin{example}
If $\theta$ is a partial action of an inverse semigroup $S$ on a locally compact Hausdorff and zero-dimensional space $X$, then the topological full pseudogroup $[[S\ltimes X]]$ is the set of all partial homeomorphisms $\varphi\colon U\to V$ of $X$ for which there are $s_1,\ldots,s_n\in S$ and compact-open $U_1,\ldots,U_n\subseteq X$ such that
\begin{enumerate}\setlength\itemsep{1ex}
\item[(i)] $U=\bigcup_{i=1}^n U_i$;
\item[(ii)] $U_i\subseteq X_{s_i^*}$ for all $i$; and
\item[(iii)] $\varphi|_{U_i}=\theta_{s_i}|_{U_i}$ for all $i$.
\end{enumerate}
\end{example}

The proposition below was proven in \cite[Corollary~3.3]{Renault2008} when one considers all open bisections instead of only compact-open ones. In any case, we provide a short and direct proof of it.
\begin{proposition}\label{prop:bisectionsemigroupfullsemigroup}
Suppose $\mathcal{G}$ is an ample groupoid. Then the homomorphism $\tau\colon\mathcal{G}^a\to [[\mathcal{G}]]$ is an isomorphism if and only if $\G$ is effective.
\end{proposition}
\begin{proof}
First suppose that $\G$ is effective, that is, $\mathcal{G}^{(0)}=\operatorname{int}(\operatorname{Iso}(\mathcal{G}))$. We need to prove that $\tau$ is injective, so assume $\tau_U=\tau_V$. Then $\tau_{V^{-1}U}=\tau_V^{-1}\circ\tau_V=\operatorname{id}_{\so(V)}$, which means that $V^{-1}U\subseteq \operatorname{Iso}(\mathcal{G})$. Since $V^{-1}U$ is open, we obtain $V^{-1}U\subseteq\mathcal{G}^{(0)}$, or equivalently $\so(V^{-1}U)=V^{-1}U$.

Moreover, from $\operatorname{id}_{\so(V)}=\tau_{V^{-1}U}$ we also have equality of the domains, $\so(V)=\so(V^{-1}U)$, which implies
\[V=V\so(V)=V\so(V^{-1}U)=VV^{-1}U\subseteq U,\]
and symmetrically we obtain $U\subseteq V$. Therefore $U=V$ and $\tau$ is injective.

Conversely, suppose $\operatorname{int}(\operatorname{Iso}(\mathcal{G}))\neq\mathcal{G}^{(0)}$. Take any nonempty compact-open bisection $U\subseteq\operatorname{int}(\operatorname{Iso}(\mathcal{G}))$ which is not contained in $\mathcal{G}^{(0)}$. Then $U\neq \so(U)$ but $\tau_U=\tau_{\so(U)}$, so $\tau$ is not injective.\qedhere
\end{proof}

Let us now summarize the connections between continuous orbit equivalence of partial actions, isomorphisms of groupoids of germs, isomorphisms of topological full pseudogroups, diagonal-preserving isomorphisms of Steinberg algebras, and consequently diagonal-preserving isomorphisms of the associated crossed products. To do so, we will use \cite[Corollary~5.8]{Steinberg2017}, which is an improvement of \cite[Theorem~3.1]{MR3848066}.

Note that each individual implication in the next theorem is valid under weaker hypotheses (e.g. \ref{theoremconnectingeverythingitemcoe}$\iff$\ref{theoremconnectingeverythingitemgroupoids} does not require that the groupoids of germs are Hausdorff).

\begin{theorem}\label{theo:isodequasetudo}
Let $R$ be an indecomposable commutative unital ring and suppose that $\theta$ and $\gamma$ are almost ample and topologically principal partial actions, and that the groupoids of germs $S\ltimes X$ and $T\ltimes Y$ are Hausdorff. Then the following are equivalent:
\begin{enumerate}[label=(\arabic*)]
\item\label{theoremconnectingeverythingitemcoe} the partial actions $\theta$ and $\gamma$ are continuously orbit equivalent;
\item\label{theoremconnectingeverythingitemgroupoids} the groupoids of germs $S\ltimes X$ and $T \ltimes Y$ are topologically isomorphic;
\item\label{theoremconnectingeverythingitemamplesemigroups} the inverse semigroups $(S\ltimes X)^a$ and $(T\ltimes Y)^a$ are topologically isomorphic;
\item\label{theoremconnectingeverythingitemfullpseudogroups} the inverse semigroups $[[S\ltimes X]]$ and $[[T\ltimes Y]]$ are isomorphic;
\item\label{theoremconnectingeverythingitemsteinbergalgebras} there exists a diagonal-preserving (ring or $R$-algebra) isomorphism between the Steinberg algebras $A_R(S\ltimes X)$ and $A_R(T \ltimes Y)$; 
\item\label{theoremconnectingeverythingitemcrossedproducts} there exists a diagonal-preserving (ring or $R$-algebra) isomorphism between the crossed products $A_R(X)\rtimes S$ and $A_R(Y)\rtimes T$.
\end{enumerate}
\end{theorem}

\begin{proof}
\ref{theoremconnectingeverythingitemcoe}$\iff$\ref{theoremconnectingeverythingitemgroupoids} follows from Theorems~\ref{theo:coegroupoidgermiso} and \ref{theo:groupoidgermscoe}.

\ref{theoremconnectingeverythingitemgroupoids}$\iff$\ref{theoremconnectingeverythingitemamplesemigroups} follows from non-commutative Stone duality: See, for example, \cite[Theorem~3.23]{MR3077869}. (Note that Hausdorff Boolean groupoids of \cite{MR3077869} corresponds to ample Hausdorff groupoids.)

\ref{theoremconnectingeverythingitemamplesemigroups}$\iff$\ref{theoremconnectingeverythingitemfullpseudogroups} follows from Proposition~\ref{prop:bisectionsemigroupfullsemigroup}.

\ref{theoremconnectingeverythingitemgroupoids}$\iff$\ref{theoremconnectingeverythingitemsteinbergalgebras} follows from  \cite[Corollary~5.8.]{Steinberg2017}.

\ref{theoremconnectingeverythingitemsteinbergalgebras}$\iff$\ref{theoremconnectingeverythingitemcrossedproducts} follows from Theorem~\ref{theo:steinbergiscrossed}.\qedhere
\end{proof} 

\section{Orbit equivalence of graphs and Leavitt path algebras}

In \cite{MR3614030}, the notion of continuous orbit equivalence for directed graphs was introduced, following Matsumoto’s notion of continuous orbit equivalence for topological Markov shifts (see \cite{MR3276420}). We will compare this notion with the continuous orbit equivalence of canonical actions of inverse semigroups associated to directed graphs. A similar study was made by Li in \cite{Li2017}, who considered the case of partial actions of free groups generated by edges of a graph. We reiterate that we do not make any assumptions on the second-countability of topological spaces, or countability of graphs.

\subsection*{Directed graphs}

A \emph{directed graph} is a tuple $E=(E^0,E^1,s,r)$, where $E^0$ is a set of \emph{vertices}, $E^1$ is a set of \emph{edges} and $s,r:E^1\to E^0$ are functions, called the \emph{source} and \emph{range}.

A \emph{path} in $E$ is a finite or infinite sequence $\mu=(\mu_i)_i=\mu_1\mu_2\cdots$, where $\mu_i\in E^1$ and $\so(\mu_{i+1})=\ra(\mu_i)$ for each $i$.

\begin{remark}
Even though every groupoid has a structure of graph, the conventions for ``concatenation'' do not agree: arrows/edges in groupoids are usually thought of functions, and thus they are read from right to left. On the other hand, the usual convention for paths in a graph is to read them from left to right. Nevertheless, this shall bring no confusion to our discussion.
\end{remark}

The \emph{length} of a finite path $\mu$ is the number $|\mu|$ of edges in $\mu$, that is, if $\mu=\mu_1\cdots\mu_n$, where $\mu_i\in E^1$, then $|\mu|=n$. Each vertex in $E^0$ is also regarded as a path of length $0$, and each edge in $E^1$ is a path of length $1$. So given an integer $n\geq 0$, we denote $E^n$ the set of paths of length $n$. The set of finite paths of $E$ will be denoted by $E^\star=\bigcup_{0\leq n<\infty}^\infty E^n$.

The length of an infinite path $\mu$ is simply $|\mu|=\infty$, and the set of all infinite paths is denoted $E^\infty$.

We extend the source map to $E^\star\cup E^\infty$, and the range map to $E^\star$ as follows: If $v\in E^0$ is a vertex, then $s(v)=r(v)=v$; If $\mu=\mu_1\cdots$ is a path (finite or infinite) of length $|\mu|\geq 1$, then $s(\mu)=s(\mu_1)$; If $\mu=\mu_1\cdots\mu_{|\mu|}$ is finite, we set $r(\mu)=r(\mu_{|\mu|})$.

Paths can be concatenated if their range and source agree, as long as we take the proper care with vertices: if $v$ is a vertex and $\mu\in E^\star$ is such that $r(\mu)=v$, then we specify that $\mu v=\mu$, and similarly if $\nu\in E^\star\cup E^\infty$ is such that $s(\rho)=v$, we set $v\rho =\rho$.

If $\mu=\mu_1\cdots\mu_{|\mu|}\in E^\star$ and $\nu=\nu_1\cdots\in E^\star\cup E^\infty$ are paths of length $\geq 1$ with $r(\mu)=s(\nu)$, we set $\mu\nu=\mu_1\cdots\mu_{|\mu|}\nu_1\cdots$.

Note that we always have $|\mu\nu|=|\mu|+|\nu|$ whenever the concatenation $\mu\nu$ is defined.

A vertex $v$ is called a \emph{sink} if $s^{-1}(v)=\varnothing$ and it is called an \emph{infinite emitter} if $|s^{-1}(v)| = \infty$. If $v\in E^0$ is either a sink or an infinite emitter then it is called \emph{singular}.

The \emph{boundary path space} of $E$ is defined as
\[\partial E:= E^\infty \cup \left\{ \mu \in E^\star : r(\mu) \text{ is singular} \right\}.\]

For a finite path $\mu \in E^\star$, we define the \emph{cylinder set} 
\[Z(\mu)= \{ \mu x \mid x \in \partial E \ \text{and} \ r(\mu)=s(x) \} \subseteq \partial E,\] 
and for a finite set $F \subseteq s^{-1}(r(\mu))$ (possibly empty), we define the \emph{generalised cylinder set}
\[Z(\mu, F)=Z(\mu)\setminus\bigcup_{e\in F}Z(\mu e) =\{\mu x\mid x\in \partial E, x_1\notin F \text{ and }r(\mu)=s(x)\} .\]
The generalised cylinder sets provide a basis of compact-open sets for a Hausdorff topology on $\partial E$ (see \cite[Theorem~2.1]{MR3119197}).
\subsection*{Graph semigroup}

We will associate an inverse semigroup $\mathcal{S}_E$ to the graph $E=(E^0,E^1,s,r)$. Let
\[\mathcal{S}_E = \left\{(\mu, \nu) \mid \mu,  \nu \in E^\star \ \text{and} \ r(\mu)=r(\nu)\right\}\cup \{0\}.\]
The product is determined by setting $0$ as a zero (absorbing) element, and
\[\ntag\label{definitionproductofgraphinversesemigroup}(\mu,\nu)(\zeta,\eta)=\begin{cases}
    (\mu, \eta \gamma),  &  \text{if } \nu=\zeta \gamma \text{ for some } \gamma \in E^\star \\
    (\mu \gamma, \eta),  &  \text{if } \zeta=\nu \gamma \text{ for some } \gamma \in E^\star  \\
    0,                   & \text{otherwise.}
\end{cases}\]
This operation makes $\mathcal{S}_E$ into an inverse semigroup, with the inverse given by
$(\mu,\nu)^*=(\nu,\mu)$ and $0^*=0$ (see \cite[Proposition~3.1]{MR1962477} for a proof). The set $E(\mathcal{S}_E)$ of idempotents of $\mathcal{S}_E$ coincides with the set of all pairs $(\mu,\mu)$, where $\mu\in E^\star$, and the zero element $0$.

Notice that the product of two pairs $(\mu,\nu),(\zeta,\eta)$ is non-zero if, and only if, 
$\nu$ is an initial segment of $\zeta$ or vice-versa. In this case, 
we say that $\nu$ and $\zeta$ are \emph{comparable}. It is easy to see that if $(\mu,\nu)\leq(\zeta,\eta)$ if and only if there is $\gamma \in E^\star$
such that $(\mu,\nu)=(\zeta\gamma,\eta\gamma)$. From this, it follows that $\mathcal{S}_E$ is a semilattice (actually, an $E^*$-unitary inverse semigroup).

We will now describe the \emph{canonical action} of $\mathcal{S}_E$ on the boundary path space $\partial E$. Given $(\mu,\nu) \in \mathcal{S}_E \setminus \{0\}$ we let 
\[\ntag\label{actiontheta}
    \theta_{(\mu, \nu)}\colon Z(\nu) \to Z(\mu),\qquad \nu x \mapsto \mu x,
\]
and $\theta_0\colon \varnothing  \rightarrow  \varnothing$ the empty map. The verification that the collection
\[\theta=\left(\left\lbrace Z(\mu)\right\rbrace_{(\mu,\nu)\in \mathcal{S}_E},\left\lbrace\theta_{(\mu,\nu)}
\right\rbrace_{(\mu,\nu)\in \mathcal{S}_E}\right)\]
is a topological (global) action of $\mathcal{S}_E$ on $\partial E$ is straightforward, by considering the different cases as in Equation \eqref{definitionproductofgraphinversesemigroup}.

Since $\mathcal{S}_E$ is a semilattice and the subsets $Z(\mu)$ are all compact -- and in particular clopen -- in $\partial E$, then the groupoid of germs $\mathcal{S}_E\ltimes\partial E$ is Hausdorff by Proposition \ref{prop:weaksemilattice}, and ample since $\partial E$ is locally compact Hausdorff and zero-dimensional.

\subsection*{The shift map and boundary path groupoid}
For each $n\in \mathbb{N}$, let $\partial E^{\geq n}=\left\{x\in\partial E : |x|\geq n\right\}$. Then $\partial E^{\geq n}=\bigcup_{\mu\in E^n}Z(\mu)$ is an open subset of $\partial E$. We define the \emph{one-sided shift} map $\sigma\colon \partial E^{\geq 1}\to \partial E$ as follows: given $x=x_1x_2\cdots\in\partial E^{\geq 1}$,
\[\ntag\label{eq:one-sidedshift}
\sigma(x)=\begin{cases}
r(x), & \text{if }|x|=1\\
x_2\cdots,& \text{if }|x|\geq 2\end{cases}\]
The $n$-fold composition $\sigma^n$ is defined on $\partial E^{\geq n}$ and we understand $\sigma^0\colon \partial E\rightarrow \partial E$ as the identity map. 
Following \cite{MR3614030}, the \emph{boundary path groupoid} of $E$ is
\begin{align*}
\G_E&= \{(x,m-n,y)\in \partial E\times\mathbb{Z}\times \partial E :
 \sigma^m(x)=\sigma^n(y)\} \\
&=\{(\mu x,|\mu|-|\nu|,\nu x) : 
\mu, \nu \in E^\star,\ x\in \partial E,\ r(\mu)=r(\nu)=s(x)\},
\end{align*} 
where the product and inverse are defined as
\[(x,k,y)(y,l,z)=(x,k+l,z),\qquad\text{and}\quad (x,k,y)^{-1}=(y,-k,x).\]
As such, $\G_E$ is a groupoid with unit space $\Go_E = \{(x,0,x):x\in \partial E\}$, which we identify with $\partial E$. To put a topology on $\G_E$, we consider finite paths $\mu,\nu \in E^\star$ with $r(\mu)=r(\nu)$, and a finite set of edges $F\subseteq s^{-1}(r(\mu))$. Then we define the sets
\[Z(\mu,\nu):=\left\{(\mu x,|\mu|-|\nu|,\nu x) : x\in s^{-1}(r(\mu))\right\}\]
and
\begin{align*}
Z(\mu,\nu,F) & := Z(\mu,\nu)\setminus\bigcup_{e\in F}Z(\mu e,\nu e) = \left\{(\mu x,|\mu|-|\nu|,\nu x) : x\in s^{-1}(r(\mu))\setminus\bigcup_{e\in F} Z(e).\right\}
\end{align*}
The collection of these sets provides a basis of compact-open bisections for a Hausdorff topology on $\G_E$ (see \cite[Proposition 2.6]{MR1432596} for more details in the case of row-finite graphs and \cite[Section 3]{MR1962477} for the general case).

\begin{proposition}\label{prop:graphgroupoidisomorphic}
The groupoid of germs $\mathcal{S}_E\ltimes \partial E$, associated to the canonical action of $\mathcal{S}_E$ on $\partial E$ and  the boundary path groupoid $\G_E$ are isomorphic as topological groupoids.
\end{proposition}
\begin{proof}
The map
\[\psi\colon\mathcal{S}_E\ast\partial E\to\G_E,\qquad \psi((\mu,\nu),x)=(\theta_{(\mu,\nu)}(x),|\mu|-|\nu|,x)\]
is a surjective semigroupoid homomorphism. Given $((\mu_i,\nu_i),x_i)\in \mathcal{S}_E\ast\partial E$ ($i=1,2$), we need to verify the equivalence
\[\psi((\mu_1,\nu_1),x_1)=\psi((\mu_2,\nu_2),x_2)\iff((\mu_1,\nu_1),x_1)\sim((\mu_2,\nu_2),x_2),\ntag\label{equationinjectivitymapgraphgroupoid}\]
where $\sim$ is the germ equivalence relation (Equation \eqref{eq:equivalencegroupoidgerms2}). This is enough, because it implies that $\psi$ factors (uniquely) to a (semi)groupoid isomorphism between $\mathcal{S}_E\ltimes\partial E$ and $\G_E$.

First, we write $x_i=\nu_ix_i'$. Then $\psi((\mu_1,\nu_1),x_1)=\psi((\mu_2,\nu_2),x_2)$ is equivalent to the following three statements (simultaneously):
\begin{enumerate}[label=(\roman*)]
\item\label{injectivitymapgraphgroupoiditem1} $\nu_1x_1'=\nu_2x_2'$;
\item\label{injectivitymapgraphgroupoiditem2} $|\mu_1|-|\nu_1|=|\mu_2|-|\nu_2|$;
\item\label{injectivitymapgraphgroupoiditem3} $\mu_1x_1'=\mu_2x_2'$.
\end{enumerate}
From items \ref{injectivitymapgraphgroupoiditem1} and \ref{injectivitymapgraphgroupoiditem3}, it follows that $\nu_1$ and $\nu_2$ are comparable, as are $\mu_1$ and $\mu_2$. Item \ref{injectivitymapgraphgroupoiditem2} then implies that either both $\nu_1$ and $\mu_1$ are subpaths of $\nu_2$ and $\mu_2$, respectively, or the reverse is true.

By symmetry, let us assume that $\nu_1$ and $\mu_1$ are subpaths of $\nu_2$ and $\mu_2$, respectively, say $\nu_2=\nu_1p$ and $\mu_2=\mu_1q$. From item \ref{injectivitymapgraphgroupoiditem1} we obtain $p x_2'=x_1'$, and thus from \ref{injectivitymapgraphgroupoiditem3}, $\mu_1px_2'=\mu_1x_1'=\mu_2x_2'=\mu_1qx_2'$, and therefore $p=q$.

In other words, these three items imply that
\begin{enumerate}[label=(\roman*)']
    \item\label{injectivitymapgraphgroupoiditem1'} $x_1=x_2$;
    \item\label{injectivitymapgraphgroupoiditem2'} $(\mu_1,\nu_1)\leq(\mu_2,\nu_2)$ or $(\mu_2,\nu_2)\leq(\mu_1,\nu_1)$
\end{enumerate}
and it is not hard to see that, conversely, items \ref{injectivitymapgraphgroupoiditem1'}' and \ref{injectivitymapgraphgroupoiditem2'}' imply \ref{injectivitymapgraphgroupoiditem1}-\ref{injectivitymapgraphgroupoiditem3}.

Finally, items \ref{injectivitymapgraphgroupoiditem1'}' and \ref{injectivitymapgraphgroupoiditem2'}' clearly imply that $((\mu_1,\nu_1),x_1)\sim((\mu_2,\nu_2),x_2)$, and the converse is also true because $\mathcal{S}_E$ is $E^*$-unitary. Therefore $\psi$ factors through a groupoid isomorphism $\Psi\colon\mathcal{S}_E\ltimes \partial E\to\G_E$. To verify that $\Psi$ is a homeomorphism, note that a basic (nonempty) open subset of $\mathcal{S}_E\ltimes\partial E$ has the form $[(\mu,\nu),Z(\nu\eta,F)]$, where $(\mu,\nu)\in\mathcal{S}_E$, $\eta\in s^{-1}(r(\nu))$ and $F$ is a finite subset of $r^{-1}(s(\eta))$. Then
\begin{align*}
    \Psi([(\mu,\nu),Z(\nu\eta,F)])=\left\{(\mu\eta x,|\mu|-|\nu|,\nu\eta x):x\in s^{-1}(r(\eta))\setminus\bigcup_{e\in F}Z(e)\right\}=Z(\mu\eta,\nu\eta,F)
\end{align*}
and these are precisely the basic open subsets of $\G_E$. Therefore, $\Psi$ is a homeomorphism.\qedhere
\end{proof}

A \emph{loop} or \emph{cycle} in a graph $E$ is a finite path $y\in E^\star$ such that $|y|\geq 1$ and $s(y)=r(y)$. An \emph{exit} of a loop $y=y_1\cdots y_{|y|}$ (where $y_i\in E^1$) is an edge $e$ for which there is $i$ such that $s(e)=s(y_i)$ and $e\neq y_i$. The graph $E$ is said to satisfy \emph{Condition (L)} if every loop has an exit.

\begin{definition}[{\cite[Definition 3.1]{MR3614030}}]
Two directed graphs $E=(E^0,E^1,r,s)$ and $F=(F^0,F^1,r,s)$  are \emph{continuously orbit equivalent} if there exists a homeomorphism $\varphi\colon \partial E \rightarrow \partial F$ together with continuous maps $k,l\colon \partial E^{\geq 1} \rightarrow \mathbb{N}$ and $k',l'\colon \partial F^{\geq 1} \rightarrow \mathbb{N}$ such that
\[\ntag\label{eq:coegraph1}
    \sigma_F^{k(x)}(\varphi(\sigma_E(x)))=\sigma_F^{l(x)}(\varphi(x)),\text{ for all }x \in \partial E^{\geq 1},
\]
and
\[\ntag\label{eq:coegraph2}
    \sigma_E^{k'(y)}(\varphi^{-1}(\sigma_F(x)))=\sigma_E^{l'(y)}(\varphi^{-1}(y)),\text{ for all }y \in \partial F^{\geq 1}.
\]
Here, $\sigma_E$ and $\sigma_F$ denote the shift maps on $E$ and $F$, respectively.
\end{definition}

The following is an analogue of \cite[Proposition~2.3]{MR3614030}. We provide a simple proof for completeness.

\begin{proposition}
Let $E=(E^0,E^1,r,s)$ be a directed graph. Then $E$ satisfies Condition (L) if and only if the canonical action $\theta$ of $\mathcal{S}_E$ on $\partial E$ is topologically principal (or equivalently, $\G_E$ is topologically principal).
\end{proposition}
\begin{proof}
Let us say that an element $x\in \partial E$ is \emph{cyclic} if there exists $x'\in E^\star$ with $|x'|\geq 1$ such that $x=x'x$, or equivalently $x=x'x'x'\cdots$, and that $x$ is \emph{periodic} if $x=\nu y$ for some $\nu\in E^\star$ and some cyclic $y$.

First suppose that $E$ satisfies Condition (L). Consider the set $X=(E^\star\cap\partial E)\cup\left\{x\in E^\infty:x\text{ is not periodic}\right\}$. Condition (L) implies that $X$ is dense in $\partial E$. We are done by proving that $X\subseteq\Lambda(\theta)$. Suppose $(\mu,\nu)\in \mathcal{S}_E$ and $x=\nu y\in Z(\nu)$ is such that $\theta_{(\mu,\nu)}(x)=x$. Let us prove that $\mu=\nu$. We have
\[\mu y=\theta_{(\mu,\nu)}(x)=x=\nu y.\ntag\label{equationproofconditionltopologicallyprincipal}\]
It follows that $\mu$ and $\nu$ are comparable, so to prove that $\mu=\nu$ it suffices to prove that $|\mu|=|\nu|$. Without loss of generality, let us assume that $\mu=\nu\mu'$ for some $\mu'$. From \eqref{equationproofconditionltopologicallyprincipal} we obtain $y=\mu'y$. However, $y$ is not cyclic, since $x$ is not periodic, so $|\mu'|=0$, and $|\mu|=|\nu\mu'|=|\nu|$. We conclude that $\theta$ is topologically principal.

Conversely, suppose $E$ does not satisfy Condition (L), and let $y$ be any loop in $E$ without exit. The element $x=yyy\cdots$ is isolated in $\partial E$, because $Z(y)=\left\{x\right\}$, and $\theta_{(y,yy)}(x)=x$. However, the only idempotent in $\mathcal{S}_E$ which is smaller than $(y,yy)$ is the zero, and $\theta_0$ is the empty function, thus $Z(y)\cap\Lambda(\theta)=\varnothing$. This proves that $\Lambda(\theta)$ is not dense in $\partial E$, therefore $\theta$ is not topologically principal.\qedhere
\end{proof}

We will now compare continuous orbit equivalence of graphs and continuous orbit equivalence of the canonical action of the associated semigroups. The following is analogue to \cite[Lemma 3.8]{Li2017}, but we do not require that the graphs satisfy Condition (L).

\begin{proposition}\label{prop:graphcoe}
Let $E=(E^0,E^1,s,r)$ and $F=(F^0,F^1,s,r)$ be directed graphs. Then $E$ and $F$ are continuously orbit equivalent if and only if the canonical actions $\theta^E$ and $\theta^F$ associated to $E$ and $F$ are continuously orbit equivalent.
\end{proposition}

\begin{proof}
Assume that $(\varphi, a, b)$ is a continuous orbit equivalence between $\theta^E$ and $\theta^F$. Given $x\in\partial E^{\geq 1}$, let us denote by $x_1\in E^1$ the first edge of $x$ (i.e., $x=x_1y$ for some $y\in\partial E$). The map $x\mapsto x_1$ is locally constant on $\partial E^{\geq 1}$ -- namely, it is the constant map $x\mapsto e$ on $Z(e)$ for each $e\in E^1$, and $\left\{Z(e):e\in E^1\right\}$ is a partition of $\partial E^{\geq 1}$.

Let $\alpha,\beta\colon\partial E^{\geq 1}\to \mathcal{S}_E$ be functions such that $a((r(x_1),x_1),x)=(\alpha(x), \beta(x))$ for all $x\in \partial E^{\geq 1}$, and define $k(x)=|\alpha(x)|$ and $l(x)=|\beta(x)|$. As $a$ is continuous, then $k$ and $l$ are continuous. Moreover, we have
\[\varphi(\sigma_E(x))=\varphi(\theta^E_{(r(x_1),x_1)}(x))=\theta_{(\alpha(x),\beta(x))}^F(\varphi(x)),\] 
which means that $\varphi(\sigma_E(x))=\alpha(x) y$ and $\varphi(x)=\beta(x) y$,  for some $y \in \partial F$. Thus
\begin{align*}
\sigma_F^{k(x)}(\varphi(\sigma_E(x)))=\sigma_F^{|\alpha(x)|}(\alpha(x)y)=y=\sigma_F^{|\beta(x)|}(\beta(x)y)=\sigma_F^{l(x)}(\varphi(x))
\end{align*}
and so \eqref{eq:coegraph1} holds. To prove \eqref{eq:coegraph2}, $k'$ and $l'$  are defined in a similar way, using $b$.

Conversely, suppose $\varphi:\partial E\to\partial F$ is a homeomorphism and that there are maps $k,l\colon \partial E^{\geq 1}\to\mathbb{N}$ satisfying, for all $x\in\partial E^{\geq 1}$,
	\[\ntag\label{equationorbitequivalenceofgraph}
	\sigma_F^{k(x)}(\varphi(\sigma_E(x)))=\sigma_F^{l(x)}(\varphi(x)).    
\]
We must show that there is a continuous function $a\colon \mathcal{S}_E\ast \partial E\to\mathcal{S}_F$ such that
	\[\ntag\label{equationorbitequicalenceofactions}
	    \varphi(\theta^E_{(\mu,\nu)}(x))=\theta^F_{a(\mu,\nu,x)}(\varphi(x)),
	\]
	for all $(\mu,\nu)\in\mathcal{S}_E$ and $x\in Z^E(\nu)$.

By Lemma \ref{lemmalocaldescriptionofcoe}, it is sufficient to prove that for all $(\mu,\nu)\in\mathcal{S}_E$ and for all $x\in Z(\nu)$, there exists an open set $U$ containing $x$ and $(\alpha,\beta)\in\mathcal{S}_F$ such that for all $\widetilde{x}\in U$,
	\[\varphi(\theta_{(\mu,\nu)}^E(\widetilde{x}))=\theta_{(\alpha,\beta)}^F(\varphi(\widetilde{x})).\]
	
Let us separate the proof in cases:	
	\begin{enumerate}
		\item Assume that $|\mu|=|\nu|=0$ (which implies that $\mu=\nu$).
	
	In this case, we simply take $U=Z^E(\nu)\cap\varphi^{-1}(Z^F(s(\varphi(x))))$. Then for all $\widetilde{x}\in U$,
	\[\varphi(\theta^E_{(\mu,\nu)}(\widetilde{x}))=\varphi(\widetilde{x})=\theta^F_{(s(\varphi(x)),s(\varphi(x)))}(\varphi(\widetilde{x})),\]
	so we are done.
	
	\item Assume that $|\mu|=0$ and $|\nu|= 1$.
	
	Let $K=k(x)$ and $L=l(x)$. For all $\widetilde{x}\in Z^E(\nu)\subseteq \partial E^{\geq 1}$, we have
	\[\theta^E_{(\mu,\nu)}(\widetilde{x})=\sigma_E(\widetilde{x})\]
	Let $U_1=Z^E(\nu)\cap k^{-1}(K)\cap l^{-1}(L)$. Then for all $\widetilde{x}\in U_1$, Equation \eqref{equationorbitequivalenceofgraph} implies that
	\[\ntag\label{equationproof}
	  \sigma_F^{K}(\varphi(\theta^E_{(\mu,\nu)}(\widetilde{x}))=\sigma_F^L(\varphi(\widetilde{x}))
	  \]
		Equation \eqref{equationproof} with $\widetilde{x}=x$ implies that there exist $(\alpha,\beta)\in\mathcal{S}_F$, with $|\alpha|=K$ and $|\beta|=L$, such that
	\[\varphi(\theta_{(\mu,\nu)}^E(x))=\theta^F_{(\alpha,\beta)}(\varphi(x)).\]
	Thus setting $U=U_1\cap\varphi^{-1}(Z^F(\nu))\cap(\varphi\circ\theta_{(\mu,\nu)}^E)^{-1}(Z^F(\mu))$, we obtain Equation \eqref{equationorbitequicalenceofactions} on $U$.
	
	\item\label{case3} Assume that $|\mu|=0$ and $|\nu|\geq 1$.
	
	Write $\nu=\nu_1\cdots\nu_{|\nu|}$, where $\nu_i\in E^1$. Notice that
	\[(\mu,\nu)=(\mu,\nu_{|\nu|})(s(\nu_{|\nu|}),\nu_{|\nu|-1})\cdots(s(\nu_3),\nu_2)(s(\nu_2),\nu_1)\]
	In other words, there are elements $e_1,\ldots,e_{|\nu|}$ of the form considered in the previous case, such that $(\mu,\nu)=e_{|\nu|}\cdots e_1$. Applying the previous case, for each $k\geq 1$ we may find a neighbourhood $U_k$ of $\theta_{e_{k-1}\cdots e_1}(x)$ (or simply $x$ in the case $k=1$) and an element $f_k\in\mathcal{S}_F$ such that
	\[\varphi\circ\theta^E_{e_k}=\theta^F_{f_k}\circ\varphi\]
	on $U_k$. Then $U=U_1\cap\bigcap_{k=2}^{|\nu|}\theta_{e_{k-1}\cdots e_1}^{-1}(U_k)$ is a neighbourhood of $x$ such that
	\[\varphi\circ\theta^E_{(\mu,\nu)}=\varphi\circ\theta^E_{e_{|\nu|}}\circ\cdots\circ\theta^E_{e_{1}}=\theta^F_{f_{|\nu|}}\circ\cdots\circ\theta^F_{f_1}\circ\varphi=\theta^F_{f_{|\nu|}\cdots f_1}\circ\varphi,\]
	since $\theta^E$ and $\theta^F$ are actions.
	
	\item\label{case4} Assume that $|\mu|\geq 1$ and $|\nu|=0$.
	
	Applying the case \ref{case3} to $(\nu,\mu)$, there exists a neighbourhood $V$ of $\theta^E_{(\mu,\nu)}(x)$ and $(\beta,\alpha)\in\mathcal{S}_F$ such that $\varphi\circ\theta^E_{(\nu,\mu)}=\theta^F_{(\beta,\alpha)}\circ\varphi$ on $V$. In other words, $\varphi\circ\theta^E_{(\mu,\nu)}=\theta^F_{(\alpha,\beta)}\circ\varphi$ on the neighbourhood $U=\theta^E_{(\nu,\mu)}(V)$ of $x$, as we wanted.
	
	\item\label{case5} Assume that $|\mu|,|\nu|\geq 1$.
	In this case, $(\mu,\nu)=(\mu,r(\mu))(r(\mu),\nu)$, so we may apply cases \ref{case3} and \ref{case4}, and proceed in a manner similar to that of case \ref{case3}.
\end{enumerate}
Since we have exhausted all possibilities for $(\mu,\nu)$, the theorem is proven.\qedhere
\end{proof}

The \emph{Leavitt path algebra} $L_R(E)$ of a directed graph $E$ with coefficients in a unital commutative ring $R$ is the $R$-algebra generated by a set $\{v \in E^0 \}$ of pairwise orthogonal idempotents and a set of variables $\left\{e, e^* : e \in E^1 \right\}$ satisfying the relations:
\begin{enumerate}[label=(\roman*)]
\item $s(e)e = e = er(e)$ for all $e \in E^1$;
\item $r(e)e^* = e^* = e^*s(e)$ for all $e \in E^1$;
\item $e^*f=\delta_{e,f}r(e)$ for all $e, f \in E^1$ (where $\delta_{x,y}$ denotes the Kronecker delta);
\item $v=\sum_{e \in s^{-1}(v)} ee^*$ whenever $v$ is not a sink nor an infinite emitter.
\end{enumerate}

The Leavitt path algebra $L_R(E)$ is isomorphic to the Steinberg algebra $A_R(\G_E)$ of the boundary path groupoid $\G_E$ (see \cite[Example~3.2]{MR3299719}). By Proposition~\ref{prop:graphgroupoidisomorphic} the groupoids $\G_E$ and $\mathcal{S}_E\ltimes \partial E$ are isomorphic, so by Theorem \ref{theo:steinbergiscrossed} we obtain the isomorphisms
\[L_R(E)\cong A_R(\G_E) \cong A_R(\mathcal{S}_E\ltimes \partial E) \cong A_R(\partial E)\rtimes \mathcal{S}_E.\]

Finally, from Propositions~\ref{prop:graphcoe} and Theorem~\ref{theo:isodequasetudo}, we obtain the following theorem:

\begin{theorem} 
Let $E$ and $F$ be directed graphs that satisfy the Condition (L) and $R$ an indecomposable commutative unital ring. Then the following are equivalent: 
\begin{enumerate}[label=(\roman*)]
\item\label{item:graph1} the graphs $E$ and $F$ are continuously orbit equivalent;
\item\label{item:graph2} the actions $\theta^E$ and $\theta^F$ are continuously orbit equivalent;
\item\label{item:graph3} $\mathcal{S}_E \ltimes \partial E$ and $\mathcal{S}_F \ltimes \partial F$  are isomorphic as topological groupoids;
\item\label{item:graph4} $\G_E$ and $\G_F$ are isomorphic as topological groupoids;
\item\label{item:graph6} there exists a diagonal-preserving isomorphism between the Steinberg algebras $A_R(\G_E)$ and $A_R(\G_F)$;
\item\label{item:graph7} there exists a diagonal-preserving isomorphism between the  skew inverse semigroup rings $A_R(\partial E)\rtimes \mathcal{S}_E$ and $A_R(\partial F)\rtimes \mathcal{S}_F$;
\item\label{item:graph8} there exists a diagonal-preserving isomorphism between the Leavitt path algebras $L_R(E)$ and $L_R(F)$.
\end{enumerate}
\end{theorem}

\bibliographystyle{amsplain}
\bibliography{references}

\providecommand{\bysame}{\leavevmode\hbox to3em{\hrulefill}\thinspace}
\providecommand{\MR}{\relax\ifhmode\unskip\space\fi MR }
\providecommand{\MRhref}[2]{%
  \href{http://www.ams.org/mathscinet-getitem?mr=#1}{#2}
}
\providecommand{\href}[2]{#2}
\begin{thebibliography}{10}

\bibitem{Abadie2004}
Fernando Abadie, \emph{On partial actions and groupoids}, Proc. Amer. Math.
  Soc. \textbf{132} (2004), no.~4, 1037--1047. \MR{2045419}

\bibitem{beutergoncalvesoinertroyer2018}
Viviane Beuter, Daniel Gonçalves, Johan Öinert, and Danilo Royer,
  \emph{Simplicity of skew inverse semigroup rings with applications to
  steinberg algebras and topological dynamics}, 2018, to appear in Forum
  Mathematicum.

\bibitem{Beuter2016}
Viviane Beuter and Daniel Gon\c{c}alves, \emph{Partial crossed products as
  equivalence relation algebras}, Rocky Mountain J. Math. \textbf{46} (2016),
  no.~1, 85--104. \MR{3506079}

\bibitem{Beuter2018}
\bysame, \emph{The interplay between {S}teinberg algebras and skew rings}, J.
  Algebra \textbf{497} (2018), 337--362. \MR{3743184}

\bibitem{Starling2018}
Tristan Bice and Charles Starling, \emph{{General non-commutative locally
  compact locally Hausdorff Stone duality}}, 2018, to appear in Advances in
  Mathematics.

\bibitem{Boava2013}
Giuliano Boava and Ruy Exel, \emph{Partial crossed product description of the
  {$C^*$}-algebras associated with integral domains}, Proc. Amer. Math. Soc.
  \textbf{141} (2013), no.~7, 2439--2451. \MR{3043025}

\bibitem{MR3614030}
Nathan Brownlowe, Toke~M. Carlsen, and Michael~F. Whittaker, \emph{Graph
  algebras and orbit equivalence}, Ergodic Theory Dynam. Systems \textbf{37}
  (2017), no.~2, 389--417. \MR{3614030}

\bibitem{MR3231479}
Alcides Buss and Ruy Exel, \emph{Inverse semigroup expansions and their actions
  on {$C^*$}-algebras}, Illinois J. Math. \textbf{56} (2012), no.~4,
  1185--1212. \MR{3231479}

\bibitem{MR2969047}
Alcides Buss, Ruy Exel, and Ralf Meyer, \emph{Inverse semigroup actions as
  groupoid actions}, Semigroup Forum \textbf{85} (2012), no.~2, 227--243.
  \MR{2969047}

\bibitem{MR3539347}
Toke~M. Carlsen and Nadia~S. Larsen, \emph{Partial actions and {KMS} states on
  relative graph {$C^*$}-algebras}, J. Funct. Anal. \textbf{271} (2016), no.~8,
  2090--2132. \MR{3539347}

\bibitem{MR3848066}
Toke~M. Carlsen and James Rout, \emph{Diagonal-preserving graded isomorphisms
  of {S}teinberg algebras}, Commun. Contemp. Math. \textbf{20} (2018), no.~6,
  1750064, 25. \MR{3848066}

\bibitem{Clark2014}
Lisa~Orloff Clark, Cynthia Farthing, Aidan Sims, and Mark Tomforde, \emph{A
  groupoid generalisation of {L}eavitt path algebras}, Semigroup Forum
  \textbf{89} (2014), no.~3, 501--517. \MR{3274831}

\bibitem{MR3299719}
Lisa~Orloff Clark and Aidan Sims, \emph{Equivalent groupoids have {M}orita
  equivalent {S}teinberg algebras}, J. Pure Appl. Algebra \textbf{219} (2015),
  no.~6, 2062--2075. \MR{3299719}

\bibitem{cordeiro2018}
Luiz~Gustavo Cordeiro, \emph{Disjoint continuous functions}, Preprint.
  \href{https://arxiv.org/abs/1711.00545v3}{arXiv:1711.00545v3}, 2018.

\bibitem{cordeirothesis}
\bysame, \emph{{Soficity and other dynamical aspects of groupoids and inverse
  semigroups}}, Ph.D. thesis, University of Ottawa, 2018, pp.~261+xii.

\bibitem{Demeneghi2017}
Paulinho Demeneghi, \emph{{The ideal structure of Steinberg algebras}},
  Preprint. \href{https://arxiv.org/abs/1710.09723v3}{arXiv:1710.09723v3},
  2018.

\bibitem{Dokuchaev2005}
Michael Dokuchaev and Ruy Exel, \emph{Associativity of crossed products by
  partial actions, enveloping actions and partial representations}, Trans.
  Amer. Math. Soc. \textbf{357} (2005), no.~5, 1931--1952. \MR{2115083}

\bibitem{MR0261565}
Robert~L. Ellis, \emph{Extending continuous functions on zero-dimensional
  spaces}, Math. Ann. \textbf{186} (1970), 114--122. \MR{0261565}

\bibitem{Exel1994}
Ruy Exel, \emph{Circle actions on {$C^*$}-algebras, partial automorphisms, and
  a generalized {P}imsner-{V}oiculescu exact sequence}, J. Funct. Anal.
  \textbf{122} (1994), no.~2, 361--401. \MR{1276163}

\bibitem{Exel1998}
\bysame, \emph{Partial actions of groups and actions of inverse semigroups},
  Proc. Amer. Math. Soc. \textbf{126} (1998), no.~12, 3481--3494. \MR{1469405}

\bibitem{Exel2008}
\bysame, \emph{Inverse semigroups and combinatorial {$C^\ast$}-algebras}, Bull.
  Braz. Math. Soc. (N.S.) \textbf{39} (2008), no.~2, 191--313. \MR{2419901}

\bibitem{MR2754831}
\bysame, \emph{Semigroupoid {$C^\ast$}-algebras}, J. Math. Anal. Appl.
  \textbf{377} (2011), no.~1, 303--318. \MR{2754831}

\bibitem{MR1703078}
Ruy Exel and Marcelo Laca, \emph{Cuntz-{K}rieger algebras for infinite
  matrices}, J. Reine Angew. Math. \textbf{512} (1999), 119--172. \MR{1703078}

\bibitem{Exel2016}
Ruy Exel and Enrique Pardo, \emph{The tight groupoid of an inverse semigroup},
  Semigroup Forum \textbf{92} (2016), no.~1, 274--303. \MR{3448414}

\bibitem{MR3699170}
Thierry Giordano, Daniel Gon\c{c}alves, and Charles Starling,
  \emph{Bratteli-{V}ershik models for partial actions of {$\mathbb{Z}$}},
  Internat. J. Math. \textbf{28} (2017), no.~10, 1750073, 17. \MR{3699170}

\bibitem{MR1363826}
Thierry Giordano, Ian~F. Putnam, and Christian~F. Skau, \emph{Topological orbit
  equivalence and {$C^*$}-crossed products}, J. Reine Angew. Math. \textbf{469}
  (1995), 51--111. \MR{1363826}

\bibitem{MR1710743}
\bysame, \emph{Full groups of {C}antor minimal systems}, Israel J. Math.
  \textbf{111} (1999), 285--320. \MR{1710743}

\bibitem{Goncalves2014c}
Daniel Gon\c{c}alves and Danilo Royer, \emph{C*-algebras associated to
  stationary ordered {B}ratteli diagrams}, Houston J. Math. \textbf{40} (2014),
  no.~1, 127--143. \MR{3210558}

\bibitem{Goncalves2014a}
\bysame, \emph{Leavitt path algebras as partial skew group rings}, Comm.
  Algebra \textbf{42} (2014), no.~8, 3578--3592. \MR{3196063}

\bibitem{1706.03628v2}
\bysame, \emph{{Simplicity and chain conditions for ultragraph Leavitt path
  algebras via partial skew group ring theory}}, 2017, preprint.
  \href{https://arxiv.org/abs/1706.03628v2}{arXiv:1706.03628v2}.

\bibitem{Goncalves2017}
\bysame, \emph{Ultragraphs and shift spaces over infinite alphabets}, Bull.
  Sci. Math. \textbf{141} (2017), no.~1, 25--45. \MR{3600124}

\bibitem{MR1455373}
John~M. Howie, \emph{Fundamentals of semigroup theory}, London Mathematical
  Society Monographs. New Series, vol.~12, The Clarendon Press, Oxford
  University Press, New York, 1995, Oxford Science Publications. \MR{1455373}

\bibitem{Mikola2017}
Mikola Khrypchenko, \emph{{Partial actions and an embedding theorem for inverse
  semigroups}}, 2017, to appear in Periodica Mathematica Hungarica.

\bibitem{MR1432596}
Alex Kumjian, David Pask, Iain Raeburn, and Jean Renault, \emph{Graphs,
  groupoids, and {C}untz-{K}rieger algebras}, J. Funct. Anal. \textbf{144}
  (1997), no.~2, 505--541. \MR{1432596}

\bibitem{MR1694900}
Mark~V. Lawson, \emph{Inverse semigroups}, World Scientific Publishing Co.,
  Inc., River Edge, NJ, 1998, The theory of partial symmetries. \MR{1694900}

\bibitem{MR3077869}
Mark~V. Lawson and Daniel~H. Lenz, \emph{Pseudogroups and their \'etale
  groupoids}, Adv. Math. \textbf{244} (2013), 117--170. \MR{3077869}

\bibitem{Li2017}
Xin Li, \emph{Partial transformation groupoids attached to graphs and
  semigroups}, Int. Math. Res. Not. IMRN (2017), no.~17, 5233--5259.
  \MR{3694599}

\bibitem{MR3789176}
\bysame, \emph{Continuous orbit equivalence rigidity}, Ergodic Theory Dynam.
  Systems \textbf{38} (2018), no.~4, 1543--1563. \MR{3789176}

\bibitem{MR1712872}
Saunders Mac~Lane, \emph{Categories for the working mathematician}, second ed.,
  Graduate Texts in Mathematics, vol.~5, Springer-Verlag, New York, 1998.
  \MR{1712872}

\bibitem{MR3276420}
Kengo Matsumoto and Hiroki Matui, \emph{Continuous orbit equivalence of
  topological {M}arkov shifts and {C}untz-{K}rieger algebras}, Kyoto J. Math.
  \textbf{54} (2014), no.~4, 863--877. \MR{3276420}

\bibitem{MR2876963}
Hiroki Matui, \emph{Homology and topological full groups of \'etale groupoids
  on totally disconnected spaces}, Proc. Lond. Math. Soc. (3) \textbf{104}
  (2012), no.~1, 27--56. \MR{2876963}

\bibitem{MR1331978}
Kevin McClanahan, \emph{{$K$}-theory for partial crossed products by discrete
  groups}, J. Funct. Anal. \textbf{130} (1995), no.~1, 77--117. \MR{1331978}

\bibitem{MR3231226}
David Milan and Benjamin Steinberg, \emph{On inverse semigroup {$C^*$}-algebras
  and crossed products}, Groups Geom. Dyn. \textbf{8} (2014), no.~2, 485--512.
  \MR{3231226}

\bibitem{MR0262402}
Walter~D. Munn, \emph{Fundamental inverse semigroups}, Quart. J. Math. Oxford
  Ser. (2) \textbf{21} (1970), 157--170. \MR{0262402}

\bibitem{Paterson1999}
Alan L.~T. Paterson, \emph{Groupoids, inverse semigroups, and their operator
  algebras}, Progress in Mathematics, vol. 170, Birkh\"auser Boston, Inc.,
  Boston, MA, 1999. \MR{1724106}

\bibitem{MR1962477}
\bysame, \emph{Graph inverse semigroups, groupoids and their
  {$C^\ast$}-algebras}, J. Operator Theory \textbf{48} (2002), no.~3, suppl.,
  645--662. \MR{1962477}

\bibitem{MR584266}
Jean Renault, \emph{A groupoid approach to {$C\sp{\ast} $}-algebras}, Lecture
  Notes in Mathematics, vol. 793, Springer, Berlin, 1980. \MR{584266}

\bibitem{Renault2008}
\bysame, \emph{Cartan subalgebras in {$C^*$}-algebras}, Irish Math. Soc. Bull.
  (2008), no.~61, 29--63. \MR{2460017}

\bibitem{MR2304314}
Pedro Resende, \emph{\'etale groupoids and their quantales}, Adv. Math.
  \textbf{208} (2007), no.~1, 147--209. \MR{2304314}

\bibitem{Sieben1997}
N\'andor Sieben, \emph{{$C^\ast$}-crossed products by partial actions and
  actions of inverse semigroups}, J. Austral. Math. Soc. Ser. A \textbf{63}
  (1997), no.~1, 32--46. \MR{1456588}

\bibitem{MR3548134}
Charles Starling, \emph{{$\rm C^*$}-algebras of {B}oolean inverse
  monoids---traces and invariant means}, Doc. Math. \textbf{21} (2016),
  809--840. \MR{3548134}

\bibitem{MR1848088}
Benjamin Steinberg, \emph{Inverse semigroup homomorphisms via partial group
  actions}, Bull. Austral. Math. Soc. \textbf{64} (2001), no.~1, 157--168.
  \MR{1848088}

\bibitem{Steinberg2010}
\bysame, \emph{A groupoid approach to discrete inverse semigroup algebras},
  Adv. Math. \textbf{223} (2010), no.~2, 689--727. \MR{2565546}

\bibitem{Steinberg2017}
\bysame, \emph{Diagonal-preserving isomorphisms of \'{e}tale groupoid
  algebras}, J. Algebra \textbf{518} (2019), 412--439. \MR{3873946}

\bibitem{MR915990}
Bret Tilson, \emph{Categories as algebra: an essential ingredient in the theory
  of monoids}, J. Pure Appl. Algebra \textbf{48} (1987), no.~1-2, 83--198.
  \MR{915990}

\bibitem{MR0419511}
Hisao Tominaga, \emph{{On $s$-unital rings}}, Math. J. Okayama Univ.
  \textbf{18} (1975/76), no.~2, 117--134.

\bibitem{MR1908882}
Jerry~E. Vaughan, \emph{Zero-dimensional spaces from linear structures}, Indag.
  Math. (N.S.) \textbf{12} (2001), no.~4, 585--596. \MR{1908882}

\bibitem{MR3119197}
Samuel B.~G. Webster, \emph{The path space of a directed graph}, Proc. Amer.
  Math. Soc. \textbf{142} (2014), no.~1, 213--225. \MR{3119197}

\end{thebibliography}

\end{document}